\documentclass[12pt]{amsart}

\usepackage[utf8]{inputenc}
\usepackage[english]{babel}
\usepackage{amsmath,amsthm,amssymb,amsfonts,mathrsfs,stmaryrd}
\usepackage{enumerate}
\usepackage{hyperref}
\usepackage[noabbrev,capitalise,nameinlink,nosort]{cleveref}
\usepackage[dvipsnames]{xcolor}
\usepackage{graphicx}
\usepackage{setspace}
\usepackage{indentfirst}
\usepackage{geometry}
\newcommand{\N}{\mathbb N}
\newcommand{\Z}{\mathbb Z}

\newcommand{\End}{\operatorname{End}}

\newcommand{\FDRep}{\operatorname{FDRep}}

\newtheorem{theorem}{Theorem}[section]
\newtheorem{lemma}[theorem]{Lemma}
\newtheorem{proposition}[theorem]{Proposition}
\newtheorem{corollary}[theorem]{Corollary}
\newtheorem{definition}[theorem]{Definition}
\newtheorem{construction}[theorem]{Construction}
\newtheorem{remark}[theorem]{Remark}
\newtheorem{example}[theorem]{Example}

\title[RFD ultragraph algebras] {Residual finite-dimensionality of ultragraph algebras via branching systems} 

\author[D. Gon\c{c}alves]{Daniel Gon\c{c}alves} \address[D. Gon\c{c}alves]{ Departamento de Matem\'atica, Universidade Federal de Santa Catarina, Florian\'opolis, SC, Brazil } \email{daemig@gmail.com} 

\author[D. Royer]{Danilo Royer} \address[D. Royer]{ Departamento de Matem\'atica, Universidade Federal de Santa Catarina, Florian\'opolis, SC, Brazil } \email{danilo.royer@ufsc.br} 

\thanks{The first author was partially supported by Conselho Nacional de
Desenvolvimento Cient\'ifico e Tecnol\'ogico (CNPq), Brazil, and Funda\c{c}\~ao
de Amparo \`a Pesquisa e Inova\c{c}\~ao do Estado de Santa Catarina (FAPESC).
The second author was partially supported by Funda\c{c}\~ao
de Amparo \`a Pesquisa e Inova\c{c}\~ao do Estado de Santa Catarina (FAPESC)}



\begin{document}

\begin{abstract}
We study residual finite-dimensionality for ultragraph algebras, both in the
algebraic and in the \(C^*\)-algebraic settings.  We introduce
graph-theoretic RFD conditions for ultragraphs, extending the conditions that
characterize RFD graph \(C^*\)-algebras.  Using the boundary ultrapath
branching system, we construct finite-dimensional branching-system
representations associated to terminal boundary sets and no-exit cycles.  These
representations are used to prove that, whenever an ultragraph satisfies the
graph-theoretic RFD conditions, its ultragraph Leavitt path algebra
\(L_K(\mathcal G)\) is RFD, for every field \(K\), and its ultragraph
\(C^*\)-algebra \(C^*(\mathcal G)\) is RFD.

For ultragraphs satisfying Condition~\emph{(RFUM2)}, we prove converses in both
settings. The analytic converse uses the groupoid model and the density of
periodic points, while the algebraic converse is proved directly by
finite-dimensional linear algebra.  Thus, for RFUM2 ultragraphs, RFD of
\(L_K(\mathcal G)\), RFD of \(C^*(\mathcal G)\), and the graph-theoretic RFD
conditions are equivalent.  This gives, in particular, a common combinatorial
description linking the algebraic and analytic theories, recovers the graph
\(C^*\)-algebra characterization, and yields an algebraic characterization for
Leavitt path algebras of graphs.  We also construct an RFD ultragraph algebra
which is genuinely outside the graph-algebra class in both settings.
\end{abstract}

\maketitle

\section{Introduction}


Residual finite-dimensionality is a finiteness property defined in terms of
finite-dimensional representations.  A \(C^*\)-algebra \(A\) is
\emph{residually finite-dimensional} (RFD) if its finite-dimensional
\(*\)-representations separate points.  This property has been studied in many
contexts in \(C^*\)-theory, including free products, approximation properties,
questions related to Kirchberg's QWEP and Connes embedding problems, the UCT
problem, and the problem of realizing \(C^*\)-algebras as subalgebras of
AF-algebras; see, for example,
\cite{ExelLoring1992,Archbold1995,Ozawa2004QWEP,Dadarlat2000NonnuclearAF}.
For graph \(C^*\)-algebras, Bellier and Shulman first characterized the
unital RFD case in \cite{BellierShulman2026}, and Bellier subsequently gave a
complete graph-theoretic characterization for countable graph \(C^*\)-algebras
in \cite{Bellier2026}.

There is a corresponding purely algebraic notion.  If \(K\) is a field, a
\(K\)-algebra \(A\) is \emph{residually finite-dimensional} if its
finite-dimensional \(K\)-representations separate points; equivalently, for
every \(0\neq a\in A\), there exist a finite-dimensional \(K\)-vector space
\(V\) and a \(K\)-algebra homomorphism
$   \pi:A\longrightarrow \End_K(V)$
such that \(\pi(a)\neq0\).   This property and related finite-dimensional
representation conditions have appeared in several parts of algebra, including
the theory of finite dual coalgebras, Hopfian and Bassian algebras, polynomial
almost identities, and FCR-type conditions; see
\cite{Reyes2024FiniteDual,Rowen2019HopfianBassian,LarsenShalev2022RFD,
KraftSmallWallach2001,Farkas2005}.

The present paper studies residual finite-dimensionality for ultragraph
algebras, both in the algebraic and in the \(C^*\)-algebraic settings.
Ultragraphs were introduced by Tomforde in \cite{Tomforde2003Ultragraph} as a
framework encompassing graph \(C^*\)-algebras and Exel--Laca algebras.  They
also provide examples which are genuinely outside the Exel--Laca and graph-algebra class.
Their purely algebraic counterparts, ultragraph Leavitt path algebras, extend
Leavitt path algebras of graphs and the algebraic analogues of Exel--Laca
algebras; see
\cite{ImanfarPourabbasLarki2020UltragraphLPA,GoncalvesRoyer2021UltragraphReduction,
DuyenGoncalvesNam}.  Ultragraphs have also been used in symbolic dynamics,
especially in the study of infinite-alphabet shift spaces and their associated
algebras; see, for example,
\cite{GoncalvesRoyer2019InfiniteAlphabetEdgeShift,
TascaGoncalves2022KMSUltragraphSinks,UltragraphGroupoids}.

A central tool in this paper is the theory of branching systems.  Branching
systems give concrete models for the generators of graph and ultragraph
algebras, producing representations on spaces of functions or on Hilbert
spaces.  They connect graph and ultragraph algebras with topics such as
wavelets, symbolic dynamics, Perron--Frobenius operators, interval maps,
permutative representations, and uniqueness theorems; see
\cite{BratteliJorgensen1997,GoncalvesRoyer2011,GoncalvesRoyer2012,
GoncalvesLiRoyer2016,GoncalvesLiRoyer2018HigherRankGraph,
GoncalvesLiRoyer2016Ultragraph,EidtGoncalvesRoyerRelativeUltragraph}.
Here we use branching systems in a different way: we construct finite
branching-system representations designed to separate prescribed elements.
This gives a branching-system approach to residual finite-dimensionality,
even in the graph case.

Although graph \(C^*\)-algebras and Leavitt path algebras often share the
same graph-theoretic characterizations, results do not automatically transfer
from one setting to the other.  The analytic and algebraic proofs usually use
different tools, and deciding whether a property has the same combinatorial
description in both settings is a separate problem.  This distinction is also
present in the current work.  In the positive direction, finite branching
systems provide a common source of finite-dimensional representations, but the
algebraic theorem uses the reduction theorem for ultragraph Leavitt path
algebras, whereas the \(C^*\)-algebraic theorem uses a Cuntz--Krieger
uniqueness theorem.  In the converse direction, the analytic argument uses the
groupoid model and the relation between RFD groupoid \(C^*\)-algebras and
periodic points \cite{ShulmanSkalski}, while the algebraic argument is based
on finite-dimensional linear algebra applied directly to representations of
ultragraph Leavitt path algebras.

We now describe the main results.  We introduce graph-theoretic RFD conditions
for ultragraphs.  These conditions require that the ultragraph have no infinite
receivers, no cycles with exits, and no infinite backward chains, and that every
vertex reach either a terminal boundary set or a cycle.  When the ultragraph is
an ordinary graph, these conditions reduce to Bellier's conditions for graph
\(C^*\)-algebras.  Our first main theorem says that, for any ultragraph
satisfying these graph-theoretic RFD conditions, the ultragraph Leavitt path
algebra \(L_K(\mathcal G)\) is RFD.  Over \(\mathbb C\), the finite
branching-system representations may be chosen to be \(*\)-representations, so
\(L_{\mathbb C}(\mathcal G)\) is \(*\)-RFD.  Using the same finite branching
systems as finite-dimensional Hilbert-space representations, together with the
general Cuntz--Krieger uniqueness theorem for ultragraph \(C^*\)-algebras, we
also prove that \(C^*(\mathcal G)\) is RFD.

We then prove converse results under Condition~\emph{(RFUM2)}.  This condition,
introduced in \cite{TascaGoncalves2022KMSUltragraphSinks}, gives a tractable
boundary ultrapath space with a locally compact Hausdorff topology and a
Deaconu--Renault groupoid model.  For RFUM2 ultragraphs, we prove that if
\(C^*(\mathcal G)\) is RFD, then \(\mathcal G\) satisfies the graph-theoretic
RFD conditions.  We also prove an algebraic converse: if \(L_K(\mathcal G)\) is
RFD, then \(\mathcal G\) satisfies the same graph-theoretic RFD conditions. Consequently,
for RFUM2 ultragraphs we obtain the equivalence

\begin{enumerate}[(i)] \item \(C^*(\mathcal G)\) is RFD; \item \(L_K(\mathcal G)\) is RFD; \item \(\mathcal G\) satisfies the graph-theoretic RFD conditions. \end{enumerate}

This gives a common combinatorial characterization linking the algebraic and
analytic theories.  Since graphs are RFUM2 ultragraphs, the graph case recovers
Bellier's \(C^*\)-algebraic characterization and gives, to the best of our
knowledge, a new algebraic characterization of RFD Leavitt path algebras of
graphs.  Moreover, every Exel--Laca algebra is isomorphic to the
\(C^*\)-algebra of an RFUM2 ultragraph.  Hence the RFUM2 equivalence theorem
also gives RFD criteria for Exel--Laca algebras in terms of the
graph-theoretic RFD conditions of their associated ultragraphs.

Finally, we construct an RFUM2 ultragraph satisfying the
graph-theoretic RFD conditions such that both \(L_K(\mathcal G)\) and
\(C^*(\mathcal G)\) are RFD, but \(L_K(\mathcal G)\) is not isomorphic to the
Leavitt path algebra of any countable graph and \(C^*(\mathcal G)\) is not
isomorphic to the graph \(C^*\)-algebra of any countable graph.  The obstruction
is given by infinitely many nonzero pairwise orthogonal central idempotents in
the algebraic setting and central projections in the \(C^*\)-setting.

The paper is organized as follows.  Section~2 recalls the necessary background
on ultragraphs, ultragraph Leavitt path algebras, ultragraph \(C^*\)-algebras,
branching systems, and residual finite-dimensionality.  We then construct the
boundary ultrapath branching system and study its tail orbits in
Section~\ref{orbits in ultragraphs}, before formulating the graph-theoretic RFD
conditions for ultragraphs in Section~\ref{sec:RFD-conditions-ultragraphs}.
We use terminal
tails and no-exit cycles to build finite branching-system representations in Section~\ref{sec:finite-dimensional-branching-representations}, and
 apply these representations to prove
the algebraic and \(C^*\)-algebraic RFD theorems in Section~\ref{sec:RFD-ultragraph-algebras}.  We prove the analytic and algebraic
converses for RFUM2 ultragraphs, together with the resulting equivalence
theorem, in Section~\ref{conversesec}. In the final section  we present an
RFD ultragraph algebra beyond the graph-algebra class.

\section{Preliminaries}

We recall the basic notation for ultragraphs, ultragraph algebras, and
branching systems.  Our conventions follow the standard references on
ultragraph \(C^*\)-algebras and ultragraph Leavitt path algebras; see
\cite{Tomforde2003Ultragraph,Tomforde2003SimplicityUltragraph,
GoncalvesRoyer2021UltragraphReduction, ImanfarPourabbasLarki2020UltragraphLPA, UltragraphGroupoids}.  For branching systems, we use the
algebraic version from \cite{GoncalvesRoyer2021UltragraphReduction} and the
analytic version from \cite{GoncalvesLiRoyer2016Ultragraph, EidtGoncalvesRoyerRelativeUltragraph}; these extend
the corresponding constructions for graph algebras studied in
\cite{GoncalvesRoyer2011,GoncalvesRoyer2012,GoncalvesLiRoyer2016}.

\subsection{Ultragraphs} An \emph{ultragraph} is a quadruple
\(\mathcal G=(G^0,\mathcal G^1,r,s)\), where \(G^0\) is a set of vertices,
\(\mathcal G^1\) is a set of edges, \(s:\mathcal G^1\to G^0\) is the source
map, and \(r:\mathcal G^1\to \mathcal P(G^0)\setminus\{\emptyset\}\) is the
range map.  Thus, unlike in an ordinary directed graph, the range of an edge
is a nonempty set of vertices.  We denote by \(\mathcal G^0\) the smallest
collection of subsets of \(G^0\) containing all singletons \(\{v\}\), all
ranges \(r(e)\), and closed under finite unions and finite intersections.
The elements of \(\mathcal G^0\) are called \emph{generalized vertices}.  We
write \(A\subseteq B\) for inclusion as subsets of \(G^0\), even when
\(A,B\in\mathcal G^0\).

A finite path in \(\mathcal G\) is either a generalized vertex
\(A\in\mathcal G^0\), regarded as a path of length zero, or a word
\(\alpha=e_1\cdots e_n\), \(n\geq 1\), such that
\(s(e_{i+1})\in r(e_i)\) for \(1\leq i<n\).  The set of finite paths is
denoted by \(\mathcal G^*\).  If \(\alpha=e_1\cdots e_n\), then
\(|\alpha|=n\), \(s(\alpha)=s(e_1)\), and \(r(\alpha)=r(e_n)\).  For
\(A\in\mathcal G^0\), we put \(|A|=0\) and \(s(A)=r(A)=A\).

A vertex \(v\in G^0\) is a \emph{sink} if \(s^{-1}(v)=\emptyset\), a
\emph{regular vertex} if \(0<|s^{-1}(v)|<\infty\), and an \emph{infinite
emitter} if \(|s^{-1}(v)|=\infty\).  We denote the set of sinks by
\(G^0_s\).  We will also use the notation
\[
   \varepsilon(A)=\{e\in\mathcal G^1:s(e)\in A\},
   \qquad A\in\mathcal G^0.
\]
Following \cite{TascaGoncalves2022KMSUltragraphSinks}, a set
\(A\in\mathcal G^0\) is a \emph{minimal infinite emitter} if
\(|\varepsilon(A)|=\infty\) and no proper subset \(B\subsetneq A\), with
\(B\in\mathcal G^0\), is either an infinite emitter or an infinite set with
\(|\varepsilon(B)|<\infty\).  Similarly, \(A\) is a \emph{minimal sink} if
\(|A|=\infty\), \(|\varepsilon(A)|<\infty\), and no proper subset
\(B\subsetneq A\), with \(B\in\mathcal G^0\), is infinite.  We say that
\(\mathcal G\) satisfies Condition \emph{(RFUM2)} if every edge range
\(r(e)\) is a finite union of minimal infinite emitters, minimal sinks, and
singletons \(\{v\}\), where \(v\) is either a sink or a regular vertex.  This
is the standing hypothesis used in the construction of ultragraph shift
spaces with sinks in \cite{TascaGoncalves2022KMSUltragraphSinks}.

A closed path is a path \(c=e_1\cdots e_n\), \(n\geq 1\), such that
\(s(e_1)\in r(e_n)\).  A cycle is a closed path \(c=e_1\cdots e_n\) for
which the vertices \(s(e_1),\ldots,s(e_n)\) are distinct.  An exit for a
closed path \(c=e_1\cdots e_n\) is either an edge
\(f\in\mathcal G^1\) such that \(s(f)\in r(e_i)\) for some \(i\) and
\(f\neq e_{i+1}\), with indices read cyclically, or a sink
\(w\in r(e_i)\) for some \(i\).  We say that \(c\) has no exits if no such
exit exists.  In particular, if \(c=e_1\cdots e_n\) is a cycle without
exits, then each \(r(e_i)\) is the singleton \(\{s(e_{i+1})\}\), again with
indices read cyclically.

\subsection{Ultragraph Leavitt path algebras and ultragraph C*-algebras} Let \(R\) be a unital commutative ring.  The ultragraph Leavitt path algebra
\(L_R(\mathcal G)\) is the universal \(R\)-algebra generated by
\(\{s_e,s_e^*:e\in\mathcal G^1\}\) and \(\{p_A:A\in\mathcal G^0\}\),
subject to the following relations:
\begin{enumerate}[(L1)]
\item \(p_\emptyset=0\), \(p_Ap_B=p_{A\cap B}\), and
\(p_{A\cup B}=p_A+p_B-p_{A\cap B}\), for all \(A,B\in\mathcal G^0\);
\item \(p_{s(e)}s_e=s_ep_{r(e)}=s_e\) and
\(p_{r(e)}s_e^*=s_e^*p_{s(e)}=s_e^*\), for all \(e\in\mathcal G^1\);
\item \(s_e^*s_f=\delta_{e,f}p_{r(e)}\), for all
\(e,f\in\mathcal G^1\);
\item \(p_v=\sum_{s(e)=v}s_es_e^*\), whenever \(v\) is a regular vertex.
\end{enumerate}
Here and throughout, we write \(p_v\) for \(p_{\{v\}}\).  This convention is
also used for other indexed families, for instance \(D_v=D_{\{v\}}\).

The ultragraph \(C^*\)-algebra \(C^*(\mathcal G)\), introduced in
\cite{Tomforde2003Ultragraph}, is the universal \(C^*\)-algebra generated by
projections \(\{p_A:A\in\mathcal G^0\}\) and partial isometries
\(\{s_e:e\in\mathcal G^1\}\), with mutually orthogonal ranges, satisfying
the Boolean relations for the projections, \(s_e^*s_e=p_{r(e)}\),
\(s_es_e^*\leq p_{s(e)}\), and
\(p_v=\sum_{s(e)=v}s_es_e^*\) for every regular vertex \(v\).  Thus the
defining relations are the \(C^*\)-algebraic analogue of the relations above,
with \(s_e^*\) denoting the Hilbert-space adjoint.

\subsection{Branching systems and invariant subsets}
\begin{definition}[Algebraic ultragraph branching system]
Let \(X\) be a set.  An \emph{algebraic \(\mathcal G\)-branching system} on
\(X\) consists of subsets \(R_e\subseteq X\), for \(e\in\mathcal G^1\),
subsets \(D_A\subseteq X\), for \(A\in\mathcal G^0\), and bijections
\(f_e:D_{r(e)}\to R_e\), such that:
\begin{enumerate}[(i)]
\item the sets \(R_e\) are pairwise disjoint;
\item \(D_\emptyset=\emptyset\), \(D_A\cap D_B=D_{A\cap B}\), and
\(D_A\cup D_B=D_{A\cup B}\), for all \(A,B\in\mathcal G^0\);
\item \(R_e\subseteq D_{s(e)}\), for all \(e\in\mathcal G^1\);
\item \(D_v=\bigsqcup_{s(e)=v}R_e\), whenever \(v\) is a regular vertex.
\end{enumerate}
\end{definition}

A branching system gives a concrete representation of \(L_R(\mathcal G)\),
where \(R\) is an unital commutative ring. Let \(M_R(X)\) be the \(R\)-module of
\(R\)-valued functions on \(X\).  The associated representation
\(\pi_X:L_R(\mathcal G)\to \End_R(M_R(X))\) is given by
\[
\pi_X(p_A)\phi=\chi_{D_A}\phi,\qquad
\pi_X(s_e)\phi=\chi_{R_e}(\phi\circ f_e^{-1}),\qquad
\pi_X(s_e^*)\phi=\chi_{D_{r(e)}}(\phi\circ f_e).
\]

Ultragraph Leavitt algebras $L_R(\mathcal{G})$ and representations induced by branching systems are well known for unital commutative rings $R$. However, since our focus in this paper are RFD algebras, from now on, we replace the unital commutative ring $R$ by a field $K$. So, the module $M_K(X)$ of the previous paragraph is in fact a $K$-vector space.

If \(X\) is finite, the representation \(\pi_X:L_K(\mathcal G)\to \End_K(M_K(X))\) is finite-dimensional.  If
\(K=\mathbb C\) and \(X\) is finite, then \(M_\mathbb C(X)\cong \ell^2(X)\),
and \(\pi_X\) is a \(*\)-representation: the projections act diagonally and
the partial isometries act as partial permutation matrices.

We shall frequently restrict branching systems to invariant subsets.  Let
\(\mathcal B=\{D_A,R_e,f_e:A\in\mathcal G^0,\ e\in\mathcal G^1\}\) be an
algebraic \(\mathcal G\)-branching system on \(X\).  A subset \(Y\subseteq X\)
is called \(\mathcal G\)-invariant if, for every \(e\in\mathcal G^1\),
\[
   f_e(Y\cap D_{r(e)})\subseteq Y
   \quad\text{and}\quad
   f_e^{-1}(Y\cap R_e)\subseteq Y.
\]

\begin{proposition}\label{restricted branching system}
Let \(\mathcal B\) be an algebraic \(\mathcal G\)-branching system on \(X\),
and let \(Y\subseteq X\) be a \(\mathcal G\)-invariant subset.  For
\(A\in\mathcal G^0\) and \(e\in\mathcal G^1\), set
\(D_A^Y=D_A\cap Y\) and \(R_e^Y=R_e\cap Y\), and let
\(f_e^Y:D_{r(e)}^Y\to R_e^Y\) be the restriction of \(f_e\).  Then
\(\{D_A^Y,R_e^Y,f_e^Y:A\in\mathcal G^0,\ e\in\mathcal G^1\}\) is an
algebraic \(\mathcal G\)-branching system on \(Y\).
\end{proposition}

\begin{proof}
The Boolean relations for the sets \(D_A^Y\), the pairwise disjointness of
the sets \(R_e^Y\), and the inclusions \(R_e^Y\subseteq D_{s(e)}^Y\) follow
immediately by intersecting the corresponding relations in \(X\) with \(Y\).
If \(v\) is regular, then
\(D_v=\bigsqcup_{s(e)=v}R_e\), and hence
\(D_v^Y=\bigsqcup_{s(e)=v}R_e^Y\).  It remains only to observe that
\(f_e^Y\) is a bijection from \(D_{r(e)}^Y\) onto \(R_e^Y\).  This follows
from the two invariance conditions: the first ensures that \(f_e\) maps
\(D_{r(e)}^Y\) into \(R_e^Y\), and the second ensures that every point of
\(R_e^Y\) has its \(f_e\)-preimage in \(D_{r(e)}^Y\).
\end{proof}

We also use analytic branching systems for ultragraph \(C^*\)-algebras, as
introduced in \cite{GoncalvesLiRoyer2016Ultragraph}.  Let \((X,\mu)\) be a
measure space.  A \(C^*\)-branching system for \(\mathcal G\) consists of
measurable sets \(D_A\) and \(R_e\), together with measurable maps
\(f_e:D_{r(e)}\to R_e\) and measurable inverses \(f_e^{-1}:R_e\to D_{r(e)}\),
such that the set-theoretic branching-system relations above hold modulo
null sets.  In addition, the pushforward measures \(\mu\circ f_e\) and
\(\mu\circ f_e^{-1}\) are required to be absolutely continuous with respect
to \(\mu\) on the corresponding domains; we denote the Radon--Nikodym
derivatives by \(\Phi_{f_e}\) and \(\Phi_{f_e^{-1}}\).

The induced representation of \(C^*(\mathcal G)\) on \(L^2(X,\mu)\) is given
by
\[
\pi(p_A)\phi=\chi_{D_A}\phi,\qquad
\pi(s_e)\phi=\chi_{R_e}\Phi_{f_e^{-1}}^{1/2}(\phi\circ f_e^{-1}),
\qquad
\pi(s_e)^*\phi=\chi_{D_{r(e)}}\Phi_{f_e}^{1/2}(\phi\circ f_e),
\]
where the functions are extended by zero outside their natural domains.  For
a finite set \(X\) with counting measure, all Radon--Nikodym derivatives are
equal to \(1\).  Thus the analytic representation is exactly the Hilbert
space version of the algebraic branching-system representation over
\(\mathbb C\).

\subsection{Residual finite-dimensionality}

We shall use residual finite-dimensionality in three closely related senses:
for algebras over a field, for complex \(*\)-algebras, and for
\(C^*\)-algebras.  The \(C^*\)-algebraic notion is classical; see, for
example, \cite{ExelLoring1992,Archbold1995}.  Its purely algebraic analogue
is also standard: an algebra is residually finite-dimensional when its
finite-dimensional representations separate points, or equivalently when it
embeds into a product of finite-dimensional algebras.  This terminology and
its equivalent formulations appear, for instance, in
\cite{Reyes2024FiniteDual,Rowen2019HopfianBassian}.

\begin{definition}[Algebraic residual finite-dimensionality]
Let \(K\) be a field, and let \(A\) be an associative \(K\)-algebra.  We say
that \(A\) is \emph{residually finite-dimensional}, or \emph{RFD}, if for
every nonzero element \(a\in A\) there exist a finite-dimensional
\(K\)-vector space \(V\) and a \(K\)-algebra homomorphism
\(\pi:A\to \End_K(V)\) such that \(\pi(a)\neq 0\).
\end{definition}

Equivalently, \(A\) is RFD if its finite-dimensional representations separate
points, that is,
\[
   \bigcap_{\pi\in \FDRep_K(A)}\ker(\pi)=\{0\},
\]
where \(\FDRep_K(A)\) denotes the class of all \(K\)-algebra homomorphisms
\(\pi:A\to \End_K(V_\pi)\) with \(\dim_K(V_\pi)<\infty\).  Since choosing a
basis identifies \(\End_K(V_\pi)\) with a full matrix algebra over \(K\), one
may equivalently use homomorphisms \(A\to M_n(K)\).  In quotient-theoretic
terms, this says that the zero ideal is the intersection of two-sided ideals
\(I\) for which \(A/I\) is finite-dimensional over \(K\).

\begin{definition}[Complex \(*\)-algebraic residual finite-dimensionality]
Let \(A\) be a complex \(*\)-algebra.  We say that \(A\) is
\emph{\(*\)-residually finite-dimensional}, or \emph{\(*\)-RFD}, if its
finite-dimensional \(*\)-representations separate points.  Equivalently, for
every \(0\neq a\in A\), there exist a finite-dimensional Hilbert space \(H\)
and a \(*\)-homomorphism \(\pi:A\to B(H)\) such that \(\pi(a)\neq 0\).
\end{definition}

For a complex \(*\)-algebra, \(*\)-RFD is generally stronger than residual
finite-dimensionality as a complex algebra, because the separating
finite-dimensional representations are required to preserve the involution.
This distinction will be relevant when we pass from algebraic branching
systems over \(\mathbb C\) to finite-dimensional \(*\)-representations of
ultragraph Leavitt path algebras.

\begin{definition}[\(C^*\)-algebraic residual finite-dimensionality]
A \(C^*\)-algebra \(A\) is \emph{residually finite-dimensional}, or
\emph{RFD}, if it has a separating family of finite-dimensional
\(*\)-representations.  Equivalently, there are positive integers \(n_i\) and
a faithful \(*\)-homomorphism
\[
   A\longrightarrow \prod_i M_{n_i}(\mathbb C).
\]
\end{definition}

\section{The boundary ultrapath branching system and tail orbits}
\label{orbits in ultragraphs}

In this section, we construct a branching system on the boundary ultrapath
space of an ultragraph.  Appropriate invariant restrictions of this branching
system will be the basic source of finite-dimensional representations in the
sequel.

Let \(A_\infty\) denote the collection of minimal infinite emitters, let
\(A_s\) denote the collection of minimal sinks, and let
\[
   \mathcal G_s^0:=\{\{v\}:v\in G^0 \text{ is a sink}\}.
\]

\begin{definition}
An element \(A\in\mathcal G^0\) is called a \emph{terminal boundary set} if
\[
   A\in \mathcal A:=A_\infty\cup A_s\cup \mathcal G_s^0.
\]
\end{definition}

\begin{remark}
A minimal infinite emitter need not be infinite as a set of vertices.  Indeed,
if \(v\) is a vertex infinite emitter, then \(\{v\}\) is a minimal infinite
emitter.  More generally, if \(A\in A_\infty\), then either \(A=\{v\}\) for a
vertex infinite emitter \(v\), or \(A\) is infinite.  On the other hand, every
minimal sink is infinite by definition.  The families \(A_\infty\), \(A_s\),
and \(\mathcal G_s^0\) are pairwise disjoint as collections of generalized
vertices: elements of \(A_\infty\) emit infinitely many edges, elements of
\(A_s\) emit only finitely many edges and are infinite sets, and elements of
\(\mathcal G_s^0\) are singleton sinks.

We shall also use the following consequence of minimality.  If
\(C\in\mathcal A\) and \(B_1,B_2\in\mathcal G^0\), then
\(C\subseteq B_1\cup B_2\) implies \(C\subseteq B_1\) or \(C\subseteq B_2\).
For singleton sinks this is immediate.  For minimal sinks it follows because
otherwise \(C\cap B_1\) and \(C\cap B_2\) would be proper generalized subsets
whose union is the infinite set \(C\), so one of them would be infinite,
contradicting minimality.  For minimal infinite emitters, if neither
\(C\subseteq B_1\) nor \(C\subseteq B_2\), then \(C\cap B_1\) and
\(C\cap B_2\) are proper generalized subsets of \(C\); since
\(\varepsilon(C)\) is infinite and
\(C=(C\cap B_1)\cup(C\cap B_2)\), one of these intersections is again an
infinite emitter, contradicting minimality.
\end{remark}

Let \(\mathcal G^{\geq 1}\) denote the set of finite paths in \(\mathcal G\)
with positive length, and let \(X_\infty\) denote the set of infinite paths in
\(\mathcal G\).  We define the finite boundary part by
\[
   X_{\mathrm{fin}}
   :=
   \{(\alpha,A):\alpha\in\mathcal G^{\geq 1},\ A\in\mathcal A,\ 
   A\subseteq r(\alpha)\}
   \cup
   \{(A,A):A\in\mathcal A\}.
\]
The boundary ultrapath space used in this paper is the set
\[
   X_{\mathcal G}:=X_\infty\cup X_{\mathrm{fin}}.
\]
At this point no topology is needed; we use only the underlying set and the
tail structure.

We define a source map \(s_X\) on \(X_{\mathcal G}\) as follows.  If
\(x=e_1e_2\cdots\in X_\infty\), then \(s_X(x)=s(e_1)\).  If
\(x=(\alpha,A)\in X_{\mathrm{fin}}\), with
\(\alpha=e_1\cdots e_n\) and \(n\geq 1\), then \(s_X(x)=s(e_1)\).  Finally,
if \(x=(A,A)\), then \(s_X(x)=A\).

For \(B\in\mathcal G^0\), we shall write \(s_X(x)\preceq B\) to mean that
\(s_X(x)\in B\) when \(s_X(x)\) is a vertex, and \(s_X(x)\subseteq B\) when
\(s_X(x)\) is a generalized vertex.  Thus, for example,
\(s_X(A,A)\preceq B\) means \(A\subseteq B\).

Let \(\beta\in\mathcal G^{\geq 1}\).  If \(x\in X_{\mathcal G}\) satisfies
\(s_X(x)\preceq r(\beta)\), we define the concatenation \(\beta x\) by the
following rules.  If \(x=e_1e_2\cdots\in X_\infty\), then
\(\beta x=\beta e_1e_2\cdots\).  If \(x=(\alpha,A)\in X_{\mathrm{fin}}\) with
\(|\alpha|\geq 1\), then $
   \beta x=(\beta\alpha,A).$
Finally, if \(x=(A,A)\), then $
   \beta x=(\beta,A).$
In each case the condition \(s_X(x)\preceq r(\beta)\) is exactly the condition
that the displayed concatenation is defined.




We now define the canonical branching system on \(X_{\mathcal G}\).  The
sets \(D_B\) are determined by the source of a boundary ultrapath, while the
sets \(R_e\) consist of those boundary ultrapaths whose first edge is \(e\).

\begin{proposition}\label{standard branch sys in G}
Let \(\mathcal G\) be an ultragraph.  For each \(e\in\mathcal G^1\), define
\[
   R_e
   =
   \{x=e_1e_2\cdots\in X_\infty:e_1=e\}
   \cup
   \{(\alpha,A)\in X_{\mathrm{fin}}:|\alpha|\geq 1
      \text{ and }\alpha_1=e\}.
\]
For each \(B\in\mathcal G^0\), define
$
   D_B=\{x\in X_{\mathcal G}:s_X(x)\preceq B\}$
  and,
for each \(e\in\mathcal G^1\), define
$ f_e:D_{r(e)}\longrightarrow R_e$ by $ f_e(x)=ex,$
where \(ex\) denotes concatenation.  Then
\[
   \{D_B,R_e,f_e:B\in\mathcal G^0,\ e\in\mathcal G^1\}
\]
is an algebraic \(\mathcal G\)-branching system on \(X_{\mathcal G}\).
\end{proposition}

\begin{proof}
The sets \(R_e\) are pairwise disjoint, because a boundary ultrapath has at
most one first edge.  Also, if \(x\in R_e\), then \(x=ey\) for some
\(y\in D_{r(e)}\), and hence \(s_X(x)=s(e)\).  Thus
\(R_e\subseteq D_{s(e)}\).

We next verify the Boolean relations among the sets \(D_B\).  Clearly
\(D_\emptyset=\emptyset\).  Let \(A,B\in\mathcal G^0\).  If \(x\in X_\infty\)
or \(x=(\alpha,C)\in X_{\mathrm{fin}}\) with \(|\alpha|\geq1\), then
\(s_X(x)\) is a vertex.  Hence, for such \(x\), membership in
\(D_A\cap D_B\) is equivalent to \(s_X(x)\in A\cap B\), and membership in
\(D_A\cup D_B\) is equivalent to \(s_X(x)\in A\cup B\).

It remains to consider points of the form \(x=(C,C)\), with
\(C\in\mathcal A\).  In this case \(x\in D_A\) precisely when
\(C\subseteq A\).  Therefore \(x\in D_A\cap D_B\) precisely when
\(C\subseteq A\cap B\), so \(D_A\cap D_B=D_{A\cap B}\).  Similarly,
\(x\in D_{A\cup B}\) precisely when \(C\subseteq A\cup B\).  By the
minimality property of terminal boundary sets recorded above, this implies
\(C\subseteq A\) or \(C\subseteq B\).  Hence \(x\in D_A\cup D_B\).  The
reverse inclusion is immediate, and so \(D_A\cup D_B=D_{A\cup B}\).

Now let \(v\in G^0\) be a regular vertex.  We prove that
\[
   D_v=\bigsqcup_{s(e)=v}R_e.
\]
If \(x\in D_v\), then \(x\) cannot be of the form \((C,C)\): indeed,
\(C\subseteq\{v\}\) would force \(C=\{v\}\), while a regular vertex is
neither a sink nor an infinite emitter.  Thus either
\(x=e_1e_2\cdots\in X_\infty\) or
\(x=(\alpha,C)\in X_{\mathrm{fin}}\) with \(|\alpha|\geq1\).  In both cases
the first edge is defined, say \(e\), and the condition \(x\in D_v\) gives
\(s(e)=v\).  Hence \(x\in R_e\) for some \(e\in s^{-1}(v)\).  The reverse
inclusion follows from \(R_e\subseteq D_{s(e)}\).  Since the sets \(R_e\) are
pairwise disjoint, the union is disjoint.

It remains to show that each \(f_e\) is a bijection from \(D_{r(e)}\) onto
\(R_e\).  By definition, \(f_e(x)=ex\), and the condition
\(x\in D_{r(e)}\) is exactly the condition that this concatenation is
defined.  Thus \(f_e(D_{r(e)})\subseteq R_e\).  Conversely, every element of
\(R_e\) begins with \(e\), and removing this first edge gives an element of
\(D_{r(e)}\).  More explicitly, the inverse \(f_e^{-1}:R_e\to D_{r(e)}\) is
given as follows.  If \(z=ey\in X_\infty\), then \(f_e^{-1}(z)=y\).  If
\(z=(e\alpha,A)\in X_{\mathrm{fin}}\) with \(|e\alpha|\geq2\), then
\(f_e^{-1}(z)=(\alpha,A)\).  Finally, if \(z=(e,A)\), then
\(f_e^{-1}(z)=(A,A)\).  Hence \(f_e\) is bijective.

Therefore all the defining conditions of an algebraic
\(\mathcal G\)-branching system are satisfied.
\end{proof}

We next introduce the tail-equivalence relation associated with the boundary
ultrapath space. For this we define the shift map, which removes the first edge whenever such an edge
is present, and fixes the terminal boundary points. We make this precise below.

Define \(\sigma:X_{\mathcal G}\to X_{\mathcal G}\) as follows.  If
\(x=e_1e_2e_3\cdots\in X_\infty\), then
\[
   \sigma(x)=e_2e_3\cdots.
\]
If \(x=(\alpha,A)\in X_{\mathrm{fin}}\), with
\(\alpha=e_1\cdots e_n\), then
\[
   \sigma(\alpha,A)=
   \begin{cases}
   (e_2\cdots e_n,A), & n\geq 2,\\
   (A,A), & n=1.
   \end{cases}
\]
Finally, for a terminal boundary point \((A,A)\in X_{\mathrm{fin}}\), we set
\[
   \sigma(A,A)=(A,A).
\]

\begin{definition}
Two elements \(x,y\in X_{\mathcal G}\) are \emph{tail equivalent} if there
exist \(m,n\in\mathbb N_0=\N\cup\{0\}\) such that
\[
   \sigma^m(x)=\sigma^n(y).
\]
In this case we write \(x\sim y\).  This is an equivalence relation.  The
equivalence class of \(x\) is denoted by \(\mathcal O(x)\), and is called the
\emph{tail orbit}, or simply the \emph{orbit}, of \(x\).
\end{definition}

\begin{remark}\label{orbits in X_G}
The following elementary facts will be used repeatedly.
\begin{enumerate}[(i)]

\item For each terminal boundary point \((A,A)\in X_{\mathrm{fin}}\), one has
\[
   \mathcal O((A,A))
   =
   \{(\alpha,A)\in X_{\mathrm{fin}}:|\alpha|\geq 1\}
   \cup
   \{(A,A)\}.
\]
Equivalently, \(\mathcal O((A,A))\) consists precisely of the finite boundary
ultrapaths whose terminal boundary set is \(A\).

\item If \((\alpha,A),(\beta,B)\in X_{\mathrm{fin}}\) and \(A\neq B\), then
\[
   \mathcal O((\alpha,A))\cap \mathcal O((\beta,B))=\emptyset.
\]
Indeed, \((\alpha,A)\sim(A,A)\) and \((\beta,B)\sim(B,B)\), while
\((A,A)\not\sim(B,B)\) when \(A\neq B\).

\item If \(x\in X_\infty\) and \(y\in X_{\mathrm{fin}}\), then
\[
   \mathcal O(x)\cap \mathcal O(y)=\emptyset.
\]
This follows because every iterate of an infinite path is still an infinite
path, whereas every finite boundary ultrapath eventually shifts to a terminal
boundary point.
\end{enumerate}
\end{remark}

\begin{proposition}\label{orbits are invariant sets}
Let \(\mathcal G\) be an ultragraph, and consider the
\(\mathcal G\)-branching system on \(X_{\mathcal G}\) constructed in
Proposition~\ref{standard branch sys in G}.  If \(Y\subseteq X_{\mathcal G}\)
is a union of tail orbits, then \(Y\) is \(\mathcal G\)-invariant.
\end{proposition}

\begin{proof}
It is enough to prove that each orbit \(\mathcal O(x)\) is
\(\mathcal G\)-invariant.  Let \(e\in\mathcal G^1\).

If \(y\in \mathcal O(x)\cap D_{r(e)}\), then \(f_e(y)=ey\).  Since
\(\sigma(ey)=y\), we have \(f_e(y)\sim y\), and hence
\(f_e(y)\in\mathcal O(x)\).  Therefore
$
   f_e(\mathcal O(x)\cap D_{r(e)})\subseteq \mathcal O(x).$

Similarly, if \(z\in \mathcal O(x)\cap R_e\), then \(z\) begins with the edge
\(e\), and \(f_e^{-1}(z)=\sigma(z)\).  Hence \(f_e^{-1}(z)\sim z\), so
\(f_e^{-1}(z)\in\mathcal O(x)\).  Thus
$
   f_e^{-1}(\mathcal O(x)\cap R_e)\subseteq \mathcal O(x).$

Therefore each orbit is \(\mathcal G\)-invariant, and any union of orbits is
\(\mathcal G\)-invariant as well.
\end{proof}

\section{Graph-theoretic RFD conditions for ultragraphs}
\label{sec:RFD-conditions-ultragraphs}

We now formulate the graph-theoretic conditions on an ultragraph that will be
used throughout the paper.  They are modeled on Bellier's characterization of
residual finite-dimensionality for graph \(C^*\)-algebras
\cite{Bellier2026}.  The main difference is that, in the ultragraph setting, the
terminal alternative must be expressed in terms of terminal boundary sets in
the boundary ultrapath space.

\begin{definition}[Infinite receivers]
Let \(\mathcal G\) be an ultragraph.  A vertex \(v\in G^0\) is called an
\emph{infinite receiver} if the set
\[
   \{e\in\mathcal G^1:v\in r(e)\}
\]
is infinite.  We say that \(\mathcal G\) has \emph{no infinite receivers} if
the above set is finite for every \(v\in G^0\).
\end{definition}

\begin{definition}[Infinite backward chains]
An \emph{infinite backward chain} in \(\mathcal G\) is an infinite sequence
of distinct edges
\[
   \cdots e_3e_2e_1
\]
such that \(s(e_i)\in r(e_{i+1})\) for every \(i\geq 1\).  We say that
\(\mathcal G\) has \emph{no infinite backward chains} if no such sequence
exists.
\end{definition}

\begin{definition}[Reaching terminal boundary sets and cycles]
Let \(\mathcal G\) be an ultragraph and let \(v\in G^0\).

We say that \(v\) \emph{reaches a terminal boundary set} if either
\(\{v\}\in\mathcal A\), or there exist a path \(\alpha\in\mathcal G^*\) with
\(|\alpha|\geq 1\) and \(s(\alpha)=v\), and a terminal boundary set
\(A\in\mathcal A\), such that \(A\subseteq r(\alpha)\).

We say that \(v\) \emph{reaches a cycle} if either \(v\) is a vertex on a
cycle, or there exists a path \(\alpha\in\mathcal G^*\), with
\(|\alpha|\geq 1\) and \(s(\alpha)=v\), such that \(r(\alpha)\) contains a
vertex lying on a cycle.
\end{definition}

\begin{remark}
The preceding definition can be read in terms of the boundary ultrapath
space.  A vertex \(v\) reaches a terminal boundary set \(A\in\mathcal A\) if
and only if there is a finite boundary ultrapath \((\alpha,A)\in
X_{\mathcal G}\) whose initial vertex is \(v\), with the length-zero case
corresponding to \(A=\{v\}\).  Similarly, \(v\) reaches a cycle if and only
if there is an infinite eventually periodic path \(x=\beta c^\infty\in
X_{\mathcal G}\), where \(c\) is a cycle and the initial vertex of \(x\) is
\(v\).  In particular, a sink or a vertex infinite emitter reaches a terminal
boundary set, and a vertex lying on a cycle reaches a cycle.
\end{remark}

\begin{definition}[Graph-theoretic RFD conditions]
We say that an ultragraph \(\mathcal G\) satisfies the
\emph{graph-theoretic RFD conditions} if the following hold:
\begin{enumerate}[(B1)]
\item \(\mathcal G\) has no infinite receivers;
\item no cycle in \(\mathcal G\) has an exit;
\item \(\mathcal G\) has no infinite backward chains;
\item every vertex \(v\in G^0\) reaches a terminal boundary set or a cycle.
\end{enumerate}
\end{definition}

\begin{remark}
Condition \emph{(B4)} says that every vertex can be continued, after
possibly following a finite path, either to a terminal boundary set
\(A\in\mathcal A\) or to a cycle.  Thus the terminal alternatives are
singleton sinks, minimal infinite emitters, and minimal sinks.  In the graph
case, these reduce to the usual terminal alternatives: sinks, infinite
emitters, and cycles.
\end{remark}

\begin{remark}
When \(\mathcal G\) is an ordinary graph, the graph-theoretic RFD conditions
above reduce to Bellier's graph conditions \cite{Bellier2026}: no infinite
receivers, no cycles with exits, no infinite backward chains, and every
vertex reaches a sink, a cycle, or an infinite emitter.
\end{remark}

\section{Finite-dimensional branching-system representations}
\label{sec:finite-dimensional-branching-representations}

In this section we construct the finite branching systems that will be used to
separate elements of ultragraph algebras.  The basic idea is to
start with the canonical branching system on \(X_{\mathcal G}\) and then
restrict it to finite tail orbits.  Under the graph-theoretic RFD conditions,
the relevant orbits are finite, and hence the associated representations are
finite-dimensional.

A \emph{terminal tail} is an element \(\xi\in X_{\mathcal G}\) of one of the
following two forms:
\begin{enumerate}[(i)]
\item \(\xi=(A,A)\), where \(A\in\mathcal A\);
\item \(\xi=c^\infty\), where \(c\) is a cycle without exits.
\end{enumerate}

Recall that if \(\xi=(A,A)\) then, by Remark~\ref{orbits in X_G},
$
   \mathcal O(\xi)
   =
   \{(\alpha,A)\in X_{\mathrm{fin}}:|\alpha|\geq 1\}
   \cup
   \{(A,A)\},$
and hence \(\mathcal O(\xi)\subseteq X_{\mathrm{fin}}\).  If
\(\xi=c^\infty\), then \(\mathcal O(\xi)\subseteq X_\infty\); its elements
are precisely the infinite paths that are eventually equal, after deleting
finitely many initial edges, to \(c^\infty\). We will use this in the proof below.

\begin{lemma}\label{finiteorbit}
Assume that \(\mathcal G\) satisfies the graph-theoretic RFD conditions, and
let \(\xi\) be a terminal tail.  Then \(\mathcal O(\xi)\) is finite.
\end{lemma}

\begin{proof}
Let \(\xi\) be a terminal tail.
Suppose, towards a contradiction, that \(\mathcal O(\xi)\) is infinite.  We
consider the tree of finite prefixes that can be attached to \(\xi\), ordered
by extension on the left.  This tree is finitely branching.  Indeed, once the
rightmost target is a vertex, the number of possible edges that can be added
on the left is finite because \(\mathcal G\) has no infinite receivers.  If
the rightmost target is a terminal boundary set \(A\), then there are only
finitely many edges \(e\) with \(A\subseteq r(e)\); otherwise any fixed
vertex of the nonempty set \(A\) would be an infinite receiver.

By K\"onig's lemma, an infinite finitely branching prefix tree contains an
infinite branch.  Such a branch either gives an infinite backward chain, or
else eventually repeats a finite cycle.  The first possibility is excluded by
the graph-theoretic RFD conditions.

It remains to rule out the possibility that infinitely many distinct prefixes
come from repeatedly going around a cycle.  Suppose such a cycle appears in
the prefix tree.  If this cycle is not the terminal cycle \(c\) in the case
\(\xi=c^\infty\), then some path must eventually leave the cycle in order to
reach the terminal tail \(\xi\).  Leaving the cycle means that, at some edge
\(e_i\) of the cycle, either one follows an edge \(f\neq e_{i+1}\) with
\(s(f)\in r(e_i)\), or one reaches a sink contained in \(r(e_i)\).  In either
case the cycle has an exit, contradicting the graph-theoretic RFD
conditions.

The only remaining possibility is that \(\xi=c^\infty\) and the cycle
repeated in the prefix tree is exactly the terminal cycle \(c\).  But this
case cannot account for infinitely many distinct elements of
\(\mathcal O(\xi)\).  Indeed, for every \(k\geq 0\), the infinite paths
\(\alpha c^k c^\infty\) and \(\alpha c^\infty\) are tail equivalent: after
deleting the initial segment \(\alpha c^k\) from the first and the initial
segment \(\alpha\) from the second, both tails become \(c^\infty\).  Hence
repeating the terminal cycle produces no new orbit elements.

Thus every possible source of infinitely many distinct prefixes has been
ruled out: an infinite branch gives an infinite backward chain, a nonterminal
cycle would have an exit, and repetitions of the terminal cycle do not
produce distinct orbit elements.  This contradicts the assumption that
\(\mathcal O(\xi)\) is infinite.

\end{proof}

\begin{construction}[Finite orbit branching systems]
\label{finite branching system for terminal tails}
Let \(\mathcal G\) be an ultragraph, let \(\xi\) be a terminal tail, and set $
   X_\xi:=\mathcal O(\xi).$
By Proposition~\ref{orbits are invariant sets}, \(X_\xi\) is invariant for
the canonical branching system on \(X_{\mathcal G}\).  Hence
Proposition~\ref{restricted branching system} gives an algebraic
\(\mathcal G\)-branching system on \(X_\xi\), denoted by
\[
   \{R_e^\xi,D_A^\xi,f_e^\xi:e\in\mathcal G^1,\ A\in\mathcal G^0\}.
\]
Explicitly, $
   D_A^\xi=\{\eta\in X_\xi:s_X(\eta)\preceq A\}$, $   R_e^\xi=\{e\eta:\eta\in D_{r(e)}^\xi\}$,
and $
   f_e^\xi:D_{r(e)}^\xi\longrightarrow R_e^\xi$ is given by $
   f_e^\xi(\eta)=e\eta.$
   
Notice that if \(\mathcal G\) satisfies the graph-theoretic RFD conditions, then
Lemma~\ref{finiteorbit} implies that \(X_\xi\) is finite.  In that case this
restricted branching system is finite, and the associated representation of
\(L_K(\mathcal G)\) is finite-dimensional.
\end{construction}

The next lemma shows that finite orbit branching systems detect the
generalized vertex projections.

\begin{lemma}[Detecting generalized vertices]\label{detecting vertices}
Assume that \(\mathcal G\) satisfies the graph-theoretic RFD conditions.  For
every nonempty \(B\in\mathcal G^0\), there exists a finite algebraic
\(\mathcal G\)-branching system \(X\) such that \(D_B\neq\emptyset\).
Consequently, the induced finite-dimensional representation of
\(L_K(\mathcal G)\) does not vanish on \(p_B\).
\end{lemma}

\begin{proof}
Let \(B\in\mathcal G^0\) be nonempty, and choose \(v\in B\).  By condition
\emph{(B4)}, the vertex \(v\) reaches either a terminal boundary set or a
cycle.  Hence there exists a terminal tail \(\xi\) and an element
\(\eta\in\mathcal O(\xi)\) whose initial source is \(v\), that is,
\(s_X(\eta)=v\).

By Construction~\ref{finite branching system for terminal tails}, the
restricted branching system on \(X_\xi=\mathcal O(\xi)\) is finite.  Since
\(s_X(\eta)=v\in B\), we have \(\eta\in D_B^\xi\).  Thus
\(D_B^\xi\neq\emptyset\) and hence the representation induced by this finite
branching system satisfies \(\pi(p_B)\neq 0\).
\end{proof}

We now construct the finite cyclic branching systems that will be used to
detect nonzero Laurent polynomials in no-exit cycles.

\begin{construction}[Finite cyclic branching systems]
\label{con:ultra-finite-cyclic}
Suppose that \(\mathcal G\) satisfies the graph-theoretic RFD conditions.  Let
\(c=e_1\cdots e_m\) be a cycle without exits, based at \(v=s(e_1)\), and let
\(\xi_c=c^\infty\).  By Lemma~\ref{finiteorbit}, the orbit $   O:=\mathcal O(\xi_c)$
is finite.  Construction~\ref{finite branching system for terminal tails}
gives an algebraic \(\mathcal G\)-branching system on \(O\), denoted by $
   \{R_e^O,D_A^O,f_e^O:e\in\mathcal G^1,\ A\in\mathcal G^0\}.$
Thus, for every \(A\in\mathcal G^0\) and every \(e\in\mathcal G^1\), $
   D_A^O=\{\eta\in O:s_X(\eta)\preceq A\},$
   $ R_e^O=\{e\eta:\eta\in D_{r(e)}^O\},$
and \(f_e^O:D_{r(e)}^O\to R_e^O\) is given by \(f_e^O(\eta)=e\eta\).

Fix \(q\geq 1\), and set $
   X_{c,q}:=O\times \mathbb Z/q\mathbb Z.$
For \(A\in\mathcal G^0\) and \(e\in\mathcal G^1\), define
\[
   D_A^{c,q}:=D_A^O\times \mathbb Z/q\mathbb Z,
   \qquad
   R_e^{c,q}:=R_e^O\times \mathbb Z/q\mathbb Z.
\]
Choose one distinguished edge of the cycle, say \(e_m\), and define
\[
   \epsilon(e)=
   \begin{cases}
   1, & e=e_m,\\
   0, & e\neq e_m,
   \end{cases}
\]
with values in \(\mathbb Z/q\mathbb Z\).  For every \(e\in\mathcal G^1\),
define
\[
   f_e^{c,q}:D_{r(e)}^{c,q}\longrightarrow R_e^{c,q}
\]
by
\[
   f_e^{c,q}(\eta,j)
   =
   \bigl(f_e^O(\eta),j+\epsilon(e)\bigr)
   =
   (e\eta,j+\epsilon(e)).
\]
The inverse map is
\[
   (z,j)\longmapsto
   \bigl((f_e^O)^{-1}(z),j-\epsilon(e)\bigr),
   \qquad z\in R_e^O.
\]
The branching-system relations follow coordinatewise from the corresponding
relations on \(O\).  Hence $
   \{R_e^{c,q},D_A^{c,q},f_e^{c,q}:e\in\mathcal G^1,\ A\in\mathcal G^0\}$ 
is a finite algebraic \(\mathcal G\)-branching system.
\end{construction}

Let \(\pi_{c,q}\) be the representation of \(L_K(\mathcal G)\) associated
with the finite cyclic branching system above, and put
$   s_c:=s_{e_1}\cdots s_{e_m}.$
Consider the subspace
\[
   W_q
   =
   \operatorname{span}_K
   \{\delta_{(c^\infty,j)}:j\in\mathbb Z/q\mathbb Z\}
   \subseteq M_K(X_{c,q}).
\]
On \(W_q\), the operator \(\pi_{c,q}(s_c)\) is the cyclic permutation
\[
   \delta_{(c^\infty,j)}
   \longmapsto
   \delta_{(c^\infty,j+1)}.
\]
Indeed, with our branching-system convention, \(\pi_{c,q}(s_e)\) prefixes
the edge \(e\).  Thus the product \(\pi_{c,q}(s_{e_1})\cdots\pi_{c,q}(s_{e_m})\) sends \((c^\infty,j)\) to \((cc^\infty,j+1)\), because the second coordinate is increased exactly once, at the distinguished edge \(e_m\). Since \(cc^\infty\) is identified with \(c^\infty\), we obtain \[ \pi_{c,q}(s_c)\delta_{(c^\infty,j)} = \delta_{(c^\infty,j+1)}. \]

\begin{lemma}[Detecting Laurent polynomials in no-exit cycles]
\label{detecting laurent polynomials}
Let \(K\) be any field, and let \(c=e_1\cdots e_m\) be a no-exit cycle in
\(\mathcal G\).  Assume that \(O=\mathcal O(c^\infty)\) is finite, and let
\(\pi_{c,q}\) be the representation associated with the finite cyclic
branching system on \(X_{c,q}=O\times\mathbb Z/q\mathbb Z\).  If
\[
   0\neq p(t)=\sum_{i=a}^{b}\lambda_i t^i\in K[t,t^{-1}],
\]
then there exists \(q\geq 1\) such that $
   \pi_{c,q}(p(s_c))\neq 0,$
where \(s_c=s_{e_1}\cdots s_{e_m}\), and negative powers are interpreted by
\(s_c^{-k}=(s_c^*)^k\).
\end{lemma}

\begin{proof}
Choose \(q>b-a\), and let \(W_q\) be as above.  By construction,
\(\pi_{c,q}(s_c)\) restricts to the cyclic permutation
$
   \delta_{(c^\infty,j)}
   \longmapsto
   \delta_{(c^\infty,j+1)}$
on \(W_q\).  Let \(P_q\) denote the corresponding \(q\)-cycle permutation
matrix.

Since \(P_q\) is invertible,
\[
   p(P_q)=P_q^a\sum_{i=a}^{b}\lambda_iP_q^{i-a}.
\]
The powers \(I,P_q,\ldots,P_q^{q-1}\) are linearly independent over \(K\),
because they have pairwise disjoint supports as permutation matrices.  Since
\(0\leq i-a<q\) for every \(i=a,\ldots,b\), the linear combination
\[
   \sum_{i=a}^{b}\lambda_iP_q^{i-a}
\]
is nonzero.  Hence \(p(P_q)\neq 0\).  Therefore the restriction of
\(\pi_{c,q}(p(s_c))\) to \(W_q\) is nonzero, and consequently
\(\pi_{c,q}(p(s_c))\neq 0\).
\end{proof}

\section{Residual finite-dimensionality of ultragraph algebras}
\label{sec:RFD-ultragraph-algebras}

We now prove the main positive results.  The algebraic result follows by
combining the reduction theorem for ultragraph Leavitt path algebras with the
finite-dimensional branching-system representations constructed in the
previous section.  The \(C^*\)-algebraic result uses the same finite branching
systems, viewed as finite-dimensional \(*\)-representations, together with
the ultragraph uniqueness theorem.

\subsection{The algebraic RFD theorem for ultragraph Leavitt path algebras}

\begin{theorem}\label{thm:algebraic-RFD-ultragraph}
Let \(\mathcal G\) be an ultragraph satisfying the graph-theoretic RFD
conditions, and let \(K\) be a field.  Then \(L_K(\mathcal G)\) is
residually finite-dimensional.
\end{theorem}

\begin{proof}
Let \(0\neq a\in L_K(\mathcal G)\).  By the reduction theorem for ultragraph
Leavitt path algebras
\cite[Theorem 3.2]{GoncalvesRoyer2021UltragraphReduction}, there exist
elements \(\mu,\nu\in L_K(\mathcal G)\), each a product of edges and ghost
edges, such that \(0\neq \mu a\nu\) and one of the following holds:
\(\mu a\nu=\lambda p_A\), for some \(0\neq\lambda\in K\) and some nonempty
\(A\in\mathcal G^0\); or \(\mu a\nu\) is a nonzero Laurent polynomial in
\(s_c\), where \(c\) is a cycle without exits.

In the first case, Lemma~\ref{detecting vertices} gives a finite
branching-system representation \(\pi\) such that \(\pi(p_A)\neq0\).  Hence
\(\pi(\mu a\nu)=\lambda\pi(p_A)\neq0\), and therefore \(\pi(a)\neq0\).

In the second case, Lemma~\ref{detecting laurent polynomials} gives a finite
cyclic branching-system representation \(\pi_{c,q}\) such that
\(\pi_{c,q}(\mu a\nu)\neq0\).  Again, it follows that
\(\pi_{c,q}(a)\neq0\).

Thus every nonzero element of \(L_K(\mathcal G)\) is detected by a
finite-dimensional representation.  Therefore \(L_K(\mathcal G)\) is RFD.
\end{proof}

Over \(\mathbb C\), the same finite branching-system representations preserve
the involution.  This gives the corresponding \(*\)-algebraic version.

\begin{corollary}\label{cor:star-RFD-ultragraph-LPA}
If \(\mathcal G\) satisfies the graph-theoretic RFD conditions, then
\(L_{\mathbb C}(\mathcal G)\) is \(*\)-RFD.
\end{corollary}

\begin{proof}
The finite branching-system representations used in the proof of
Theorem~\ref{thm:algebraic-RFD-ultragraph} are partial permutation
representations on finite sets.  After identifying the complex function space
on a finite set \(X\) with \(\ell^2(X)\), the projections \(p_A\) act as
orthogonal diagonal projections and the elements \(s_e\) act as partial
permutation matrices.  Hence these representations are
\(*\)-representations.  Since they separate the points of
\(L_{\mathbb C}(\mathcal G)\), the algebra is \(*\)-RFD.
\end{proof}

\subsection{The \texorpdfstring{\(C^*\)}{C*}-algebraic RFD theorem for ultragraphs}

We now pass to ultragraph \(C^*\)-algebras.  A finite algebraic branching
system may be regarded as an analytic branching system by equipping the
finite set with counting measure.  The Radon--Nikodym derivatives are then
identically \(1\), so the associated representation is a finite-dimensional
\(*\)-representation of \(C^*(\mathcal G)\).

\begin{theorem}\label{thm:Cstar-RFD-ultragraph}
Let \(\mathcal G\) be an ultragraph satisfying the graph-theoretic RFD
conditions.  Then \(C^*(\mathcal G)\) is residually finite-dimensional.
\end{theorem}

\begin{proof}
For each vertex \(v\in G^0\), choose a finite-orbit branching-system
representation
\[
   \widetilde\pi_v:C^*(\mathcal G)\longrightarrow B(H_v)
\]
such that \(\widetilde\pi_v(p_v)\neq0\).  This is possible by
Lemma~\ref{detecting vertices}, applied to the generalized vertex
\(\{v\}\).  Moreover, if \(A\in\mathcal G^0\) is nonempty and \(v\in A\),
then the same representation satisfies \(\widetilde\pi_v(p_A)\neq0\).

For each cycle \(c\) without exits and each \(q\geq1\), let
\[
   \widetilde\pi_{c,q}:C^*(\mathcal G)\longrightarrow B(H_{c,q})
\]
be the finite-dimensional representation associated to the cyclic branching
system \(X_{c,q}\), where \(H_{c,q}=\ell^2(X_{c,q})\), see Lemma~\ref{detecting laurent polynomials}.

Let \(\Pi\) be the direct sum of all representations \(\widetilde\pi_v\),
with \(v\in G^0\), and all representations \(\widetilde\pi_{c,q}\), where
\(c\) ranges over the cycles without exits and \(q\geq1\).  We verify the
hypotheses of the general Cuntz--Krieger uniqueness theorem for ultragraph
\(C^*\)-algebras \cite[Theorem 7.4]{GoncalvesLiRoyer2016Ultragraph}.

First, if \(A\in\mathcal G^0\) is nonempty, choose \(v\in A\).  The summand
\(\widetilde\pi_v\) satisfies \(\widetilde\pi_v(p_A)\neq0\).  Hence
\(\Pi(p_A)\neq0\).

Second, let \(c=e_1\cdots e_m\) be a simple cycle without exits, based at
\(v=s(e_1)\).  For each \(q\geq 1\), the construction of
\(\widetilde\pi_{c,q}\) ensures that, on the subspace
\[
\ell^2(\{(\xi_c,j):j\in\Z/q\Z\})\subseteq \widetilde\pi_{c,q}(p_v)H_{c,q},
\]
the operator \(\widetilde\pi_{c,q}(s_c)\) is the cyclic permutation of order
\(q\).  Thus the spectrum of \(\widetilde\pi_{c,q}(s_c)\) contains the set
of \(q\)-th roots of unity.  Since all these representations occur as direct
summands of \(\Pi\), the spectrum of \(\Pi(s_c)\) contains the \(q\)-th
roots of unity for every \(q\geq 1\).  Therefore \(\sigma(\Pi(s_c))\)
contains all roots of unity.  Since the spectrum is closed, it follows that
\(\mathbb T\subseteq\sigma(\Pi(s_c))\).

The two hypotheses of the ultragraph version of Szyma\'nski's theorem, \cite{GoncalvesLiRoyer2016Ultragraph} (see \cite{Szymanski2002} for the graph case), are
therefore satisfied.  Hence \(\Pi\) is faithful.  Since \(\Pi\) is a direct
sum of finite-dimensional representations, the finite-dimensional summands
separate the points of \(C^*(\mathcal G)\). Thus \(C^*(\mathcal G)\) is
RFD.
\end{proof}

\section{Converse results}\label{conversesec}

\subsection{The analytic converse}

In this subsection we prove an analytic converse to the RFD theorem for
ultragraph \(C^*\)-algebras satisfying Condition~\emph{(RFUM2)}.  This
condition was introduced in \cite{TascaGoncalves2022KMSUltragraphSinks} as a
generalization of Condition~\emph{(RFUM)} from
\cite{GoncalvesRoyer2019InfiniteAlphabetEdgeShift}.  The additional
RFUM2 hypothesis allows us to use the topology and groupoid model of the
boundary ultrapath space.

Let \(X_{\mathcal G}\) be the boundary ultrapath space from
Section~\ref{orbits in ultragraphs}.  For ultragraphs satisfying
Condition~\emph{(RFUM2)}, the space \(X_{\mathcal G}\) is locally compact
Hausdorff and has a basis of compact-open cylinders
\cite{TascaGoncalves2022KMSUltragraphSinks}.  Moreover, the shift on
\(X_{\mathcal G}\) gives rise to the Deaconu--Renault groupoid
\[
   \mathfrak G_{\mathcal G}
   =
   \{(x,m-n,y):x,y\in X_{\mathcal G},\ m,n\in\mathbb N_0,\ 
   \sigma^m(x)=\sigma^n(y)\},
\]
with the usual convention that the iterates are taken on their natural
domains.  The unit space of \(\mathfrak G_{\mathcal G}\) is identified with
\(X_{\mathcal G}\), and the associated groupoid \(C^*\)-algebra is
isomorphic to \(C^*(\mathcal G)\).

For \(x\in X_{\mathcal G}\), its groupoid orbit is
\[
   \mathcal O(x)
   =
   \{y\in X_{\mathcal G}:\text{ there exist }m,n\in\mathbb N_0
   \text{ such that } \sigma^m(y)=\sigma^n(x)\}.
\]
Thus \(y\in\mathcal O(x)\) precisely when \(x\) and \(y\) have a common
forward shift; equivalently, after deleting finitely many initial edges from
each, the same boundary ultrapath is obtained.

Following \cite{ShulmanSkalski}, we use the following terminology.

\begin{definition}
A point \(x\in X_{\mathcal G}\) is called \emph{periodic} if its orbit
\(\mathcal O(x)\) in \(\mathfrak G_{\mathcal G}^{(0)}\) is finite.
\end{definition}

We shall use the following consequence of \cite{ShulmanSkalski}: if the
\(C^*\)-algebra of an amenable \'etale groupoid is RFD, then its unit space
has a dense set of periodic points.  Applied to the groupoid
\(\mathfrak G_{\mathcal G}\), this gives the implication
\[
   C^*(\mathcal G)\text{ is RFD}
   \quad\Longrightarrow\quad
   \text{the periodic points are dense in }X_{\mathcal G}.
\]

\begin{theorem}[Converse for RFUM2 ultragraphs]\label{analiticconverse}
Let \(\mathcal G\) be an ultragraph satisfying Condition~\emph{(RFUM2)}.  If
\(C^*(\mathcal G)\) is RFD, then \(\mathcal G\) satisfies the
graph-theoretic RFD conditions.
\end{theorem}

\begin{proof}
We prove the contrapositive, showing that the failure of any one of the four
graph-theoretic RFD conditions produces a nonempty open subset of
\(X_{\mathcal G}\) containing no periodic points.

First suppose that \(\mathcal G\) has an infinite receiver.  Thus there is a
vertex \(v\in G^0\) and infinitely many distinct edges
\(e_1,e_2,\ldots\) such that \(v\in r(e_i)\) for every \(i\).  Let
\[
   X_v:=\{x\in X_{\mathcal G}:s_X(x)\preceq\{v\}\}.
\]
By the RFUM2 topology, \(X_v\) is a nonempty open subset of
\(X_{\mathcal G}\).  If \(x\in X_v\), then \(e_i x\) is defined for every
\(i\), and all the points \(e_i x\) belong to the groupoid orbit of \(x\).
Since the first edges are distinct, these points are distinct.  Hence every
point of \(X_v\) has infinite orbit, so \(X_v\) contains no periodic points.
This contradicts density of periodic points.  Therefore no vertex can be an
infinite receiver.

Next suppose that \(c=e_1\cdots e_n\) is a cycle with an exit.  We consider
the two possible types of exits.  First assume that the exit is an edge
\(f\neq e_{i+1}\), with \(s(f)\in r(e_i)\).  Put
\(\rho=e_1\cdots e_i f\).  The cylinder
\[
   D_{(\rho,r(f))}
   =
   \{x\in X_{\mathcal G}:x=\rho y \text{ and } s_X(y)\preceq r(f)\}
\]
is a nonempty open subset of \(X_{\mathcal G}\).  If
\(x\in D_{(\rho,r(f))}\), then, for each \(k\geq1\), the point \(c^k x\) is
defined and belongs to the orbit of \(x\).  These points are all distinct,
and hence every point of \(D_{(\rho,r(f))}\) has infinite orbit.  Thus this
open set contains no periodic points.

Now assume that the exit is a sink \(w\in r(e_i)\).  Let
\(\rho=e_1\cdots e_i\).  The finite boundary ultrapath \((\rho,\{w\})\) is
an isolated point of \(X_{\mathcal G}\), so \(\{(\rho,\{w\})\}\) is open.
But the points \((c^k\rho,\{w\})\), \(k\geq1\), are all defined, distinct,
and belong to the same orbit.  Hence \((\rho,\{w\})\) is not periodic, and
again we have a nonempty open set with no periodic points.  Therefore no
cycle can have an exit.

Now suppose that \(\mathcal G\) has an infinite backward chain
\(e_0 e_1 e_2\ldots\), with \(s(e_i)\in r(e_{i+1})\) for every \(i\geq0\).
Consider the nonempty cylinder \(D_{(e_0,r(e_0))}\).  If
\(x\in D_{(e_0,r(e_0))}\), then \(e_k e_{k-1}\cdots e_1x\) is defined for
every \(k\geq1\).  These points are distinct and all belong to the groupoid
orbit of \(x\).  Hence every point of \(D_{(e_0,r(e_0))}\) has infinite
orbit.  This gives a nonempty open set with no periodic points,
contradicting density.

Finally suppose that there is a vertex \(v\in G^0\) which reaches neither a
terminal boundary set nor a cycle.  Consider again the nonempty open set
\(X_v=\{x\in X_{\mathcal G}:s_X(x)\preceq\{v\}\}\).  Let \(x\in X_v\).
Since \(v\) does not reach a terminal boundary set, \(x\) cannot be a finite
boundary ultrapath.  Thus \(x\) is an infinite path.  If \(x\) were
eventually periodic, then some vertex reached along \(x\) would lie on a
cycle, and hence \(v\) would reach a cycle, contrary to assumption.  Hence
\(x\) is not eventually periodic.  Therefore the shifts
\(x,\sigma(x),\sigma^2(x),\ldots\) are infinitely many distinct points in the
groupoid orbit of \(x\).  Thus \(x\) is not periodic, and \(X_v\) is a
nonempty open set with no periodic points.

In each case, the failure of one of the graph-theoretic RFD conditions
contradicts the density of periodic points.  Therefore all four
graph-theoretic RFD conditions hold.
\end{proof}

\subsection{The algebraic converse}

We now prove the converse in the algebraic setting.  Throughout this
subsection, \(K\) is an arbitrary field and
\(\mathcal G=(G^0,\mathcal G^1,r,s)\) is an ultragraph satisfying
Condition~\emph{(RFUM2)}.  If $
   \pi:L_K(\mathcal G)\longrightarrow \End_K(V)
$
is a finite-dimensional representation, we write
\[
   P_A:=\pi(p_A),\qquad S_e:=\pi(s_e),\qquad S_e^*:=\pi(s_e^*).
\]
For \(A\in\mathcal G^0\), set \(V_A:=P_A(V)\).

We shall repeatedly use the following elementary facts.  If \(A\subseteq B\),
then \(V_A\subseteq V_B\).  If \(A\cap B=\emptyset\), then
\(V_A\cap V_B=\{0\}\).  For every edge \(e\), the map \(S_e\) is injective
on \(V_{r(e)}\), because \(S_e^*S_e=P_{r(e)}\).  More generally, if
\(\alpha=e_1\cdots e_n\) is a finite path, then \(S_\alpha\) is injective
on \(V_{r(\alpha)}\), since \(S_\alpha^*S_\alpha=P_{r(\alpha)}\).  We also
use that \(p_A\neq0\) in \(L_K(\mathcal G)\) whenever
\(\emptyset\neq A\in\mathcal G^0\).

\begin{proposition}\label{prop:algebraic-converse-RFUM2}
Let \(\mathcal G\) be an ultragraph satisfying Condition~\emph{(RFUM2)} and let $K$ be a field.  If
\(L_K(\mathcal G)\) is RFD, then \(\mathcal G\) satisfies the
graph-theoretic RFD conditions.
\end{proposition}

\begin{proof}
We prove the contrapositive.  In each case, we find a nonzero projection
\(p_A\in L_K(\mathcal G)\) that is killed by every finite-dimensional
representation.  This prevents \(L_K(\mathcal G)\) from being RFD.

First suppose that \(\mathcal G\) has an infinite receiver.  Thus there is a
vertex \(v\in G^0\) and infinitely many distinct edges \(e_1,e_2,\ldots\)
such that \(v\in r(e_i)\) for every \(i\).  Let
\(\pi:L_K(\mathcal G)\to\End_K(V)\) be a finite-dimensional representation.

For each \(i\), we have \(V_v\subseteq V_{r(e_i)}\), and hence \(S_{e_i}\)
is injective on \(V_v\).  Moreover, for every \(N\geq1\), the sum
\[
   S_{e_1}(V_v)+\cdots+S_{e_N}(V_v)
\]
is direct.  Indeed, if \(\sum_{i=1}^N S_{e_i}(\xi_i)=0\), with
\(\xi_i\in V_v\), then applying \(S_{e_j}^*\) gives
\[
   0=S_{e_j}^*\sum_{i=1}^N S_{e_i}(\xi_i)
    =P_{r(e_j)}(\xi_j)
    =\xi_j.
\]
Thus each \(\xi_j\) is zero, and the sum is direct.

Therefore \(V\) contains, for every \(N\), a direct sum of \(N\) copies of
\(V_v\).  Since \(V\) is finite-dimensional, this is possible only if
\(V_v=0\).  Hence \(\pi(p_v)=0\) for every finite-dimensional representation
\(\pi\).  Since \(p_v\neq0\) in \(L_K(\mathcal G)\), the algebra
\(L_K(\mathcal G)\) is not RFD.

Next suppose that \(\mathcal G\) has an infinite backward chain.  Thus there
are distinct edges \(e_0,e_1,e_2,\ldots\) such that
\(s(e_i)\in r(e_{i+1})\) for all \(i\geq0\).  For \(n\geq0\), put $
   \alpha_n=e_n e_{n-1}\cdots e_0.$
Then \(\alpha_n\) is a path and \(r(\alpha_n)=r(e_0)\).  For each \(n\),
the map \(S_{\alpha_n}\) is injective on \(V_{r(e_0)}\).  Moreover, for each
\(N\geq0\), the sum
\[
   S_{\alpha_0}(V_{r(e_0)})+\cdots+S_{\alpha_N}(V_{r(e_0)})
\]
is direct.  Indeed, if
\(\sum_{n=0}^N S_{\alpha_n}(\xi_n)=0\), with
\(\xi_n\in V_{r(e_0)}\), then applying \(S_{\alpha_j}^*\) gives
\(\xi_j=0\).  This follows from
\(S_{\alpha_j}^*S_{\alpha_n}=0\) for \(j\neq n\), because the paths
\(\alpha_j\) and \(\alpha_n\) have distinct first edges, and from
\(S_{\alpha_j}^*S_{\alpha_j}=P_{r(e_0)}\).  Thus \(V\) contains arbitrarily
large direct sums of copies of \(V_{r(e_0)}\), and finite dimensionality
forces \(V_{r(e_0)}=0\).  Therefore every finite-dimensional representation
kills \(p_{r(e_0)}\).  Since \(r(e_0)\neq\emptyset\), the projection
\(p_{r(e_0)}\) is nonzero in \(L_K(\mathcal G)\).  Hence
\(L_K(\mathcal G)\) is not RFD.

Now suppose that \(\mathcal G\) has a cycle with an exit.  By cyclically
relabeling the cycle, we may write it as \(d=e_1\cdots e_n\), with an exit
at the last edge \(e_n\).  Thus either there is an edge \(f\neq e_1\) with
\(s(f)\in r(e_n)\), or there is a sink \(w\in r(e_n)\).  Put
\(u_i=s(e_i)\), so that \(u_1=s(e_1)\in r(e_n)\).

The path \(d=e_1\cdots e_n\) starts at \(u_1\) and has range \(r(e_n)\).
Since \(P_{u_1}S_d=S_d\), we have \(S_d(V)\subseteq V_{u_1}\), and in
particular $ S_d(V_{r(e_n)})\subseteq V_{u_1}.$
Also \(V_{u_1}\subseteq V_{r(e_n)}\), because \(u_1\in r(e_n)\).  Hence $
   S_d(V_{r(e_n)})\subseteq V_{u_1}\subseteq V_{r(e_n)}.$
Since \(S_d\) is injective on \(V_{r(e_n)}\), the spaces
\(S_d(V_{r(e_n)})\) and \(V_{r(e_n)}\) have the same dimension.  It follows
that $
   S_d(V_{r(e_n)})=V_{u_1}=V_{r(e_n)}.$

On the other hand, since \(d\) begins with \(e_1\), we have $
   S_d(V_{r(e_n)})\subseteq S_{e_1}(V_{r(e_1)}),$
while \(S_{e_1}(V_{r(e_1)})\subseteq V_{u_1}\).  Therefore $
   S_{e_1}(V_{r(e_1)})=V_{u_1}.$
Thus, in any finite-dimensional representation, the cycle edge \(e_1\)
already fills the whole space \(V_{u_1}\).

We now use the exit.  If the exit is a sink \(w\in r(e_n)\), then
\(w\neq u_1\), since \(u_1\) emits \(e_1\).  Hence
\(V_w\subseteq V_{r(e_n)}=V_{u_1}\), but
\(V_w\cap V_{u_1}=\{0\}\).  Therefore \(V_w=\{0\}\).  Thus every
finite-dimensional representation kills \(p_w\), while \(p_w\neq0\).

If the exit is an edge \(f\neq e_1\) with \(s(f)\in r(e_n)\), put
\(w=s(f)\).  If \(w\neq u_1\), then the same argument as above gives
\(V_w=\{0\}\).  Consequently \(S_f=P_wS_f=0\), and hence
\(P_{r(f)}=S_f^*S_f=0\).
It remains to consider the case \(w=u_1\).  Then
\(S_f(V_{r(f)})\subseteq V_{u_1}\).  But we have just shown that
\(V_{u_1}=S_{e_1}(V_{r(e_1)})\).  Since \(f\neq e_1\), the ultragraph
relations give \(S_{e_1}^*S_f=0\).  Hence
\[
   S_{e_1}(V_{r(e_1)})\cap S_f(V_{r(f)})=\{0\}.
\]
Thus \(S_f(V_{r(f)})\) is a subspace of \(V_{u_1}\) with zero intersection
with \(V_{u_1}\).  Hence \(S_f(V_{r(f)})=0\), so \(S_f=0\), and again
\(P_{r(f)}=S_f^*S_f=0\).

Therefore, in all cases, some nonzero projection, either \(p_w\) or
\(p_{r(f)}\), is killed by every finite-dimensional representation.  Hence
\(L_K(\mathcal G)\) is not RFD.

Finally suppose that there exists a vertex \(v\in G^0\) which reaches
neither a terminal boundary set nor a cycle.  We show that every
finite-dimensional representation kills \(p_v\).  Suppose, for a
contradiction, that \(V_v\neq0\).  Since \(v\) does not reach a terminal
boundary set, \(v\) is neither a sink nor an infinite emitter.  Thus \(v\) is
regular.  By the Cuntz--Krieger relation,
\[
   P_v=\sum_{s(e)=v}S_eS_e^*.
\]
Since \(V_v\neq0\), there exists an edge \(e_1\) with \(s(e_1)=v\) such that
\(S_{e_1}S_{e_1}^*(V)\neq0\).  Hence \(V_{r(e_1)}\neq0\).

Because \(\mathcal G\) satisfies Condition~\emph{(RFUM2)}, the range
\(r(e_1)\) is a finite union of minimal infinite emitters, minimal sinks,
singleton sinks, and singleton regular vertices.  Since \(v\) reaches no
terminal boundary set, \(r(e_1)\) cannot contain a minimal infinite emitter,
a minimal sink, or a singleton sink.  Hence, in the present situation,
\(r(e_1)\) is a finite union of singleton regular vertices.  Since
\(V_{r(e_1)}\neq0\), at least one of these vertices, say \(v_1\), satisfies
\(V_{v_1}\neq0\).

Repeating the argument, we obtain an infinite sequence of vertices
\(v=v_0,v_1,v_2,\ldots\), with \(V_{v_i}\neq0\) for all \(i\), and such that
\(v_i\) reaches \(v_{i+1}\) by one edge.  Since \(v\) reaches no cycle, the
vertices \(v_i\) are pairwise distinct; otherwise a repeated vertex would
produce a cycle reachable from \(v\).  Therefore, for every \(N\), the sum
\[
   V_{v_0}+\cdots+V_{v_N}
\]
is direct, because the projections \(P_{v_i}\) are pairwise orthogonal.
Thus \(V\) contains arbitrarily large direct sums of nonzero subspaces,
contradicting finite dimensionality.  Hence \(V_v=0\).  Therefore every
finite-dimensional representation kills \(p_v\), while \(p_v\neq0\) in
\(L_K(\mathcal G)\).  So \(L_K(\mathcal G)\) is not RFD.
\end{proof}

\subsection{Equivalence for RFUM2 ultragraphs and graphs}

We finish this section by recording the equivalence between the algebraic
and \(C^*\)-algebraic residual finite-dimensionality results in the RFUM2
setting.  

\begin{theorem}[Equivalence for RFUM2 ultragraphs]
\label{thm:RFUM2-equivalence-algebraic-Cstar-RFD}
Let \(\mathcal G\) be an ultragraph satisfying Condition~\emph{(RFUM2)}, and
let \(K\) be a field.  The following are equivalent:
\begin{enumerate}[(i)]
\item \(C^*(\mathcal G)\) is RFD;
\item \(L_K(\mathcal G)\) is RFD;
\item \(\mathcal G\) satisfies the graph-theoretic RFD conditions.
\end{enumerate}
\end{theorem}

\begin{proof}
The implications \((iii)\Rightarrow(ii)\) and \((iii)\Rightarrow(i)\) are
Theorems~\ref{thm:algebraic-RFD-ultragraph} and
\ref{thm:Cstar-RFD-ultragraph}, respectively.  Conversely,
\((i)\Rightarrow(iii)\) follows from Theorem~\ref{analiticconverse}, and
\((ii)\Rightarrow(iii)\) follows from
Proposition~\ref{prop:algebraic-converse-RFUM2}.  Hence the three conditions
are equivalent.
\end{proof}

Since every graph may be regarded as an ultragraph satisfying Condition~\emph{(RFUM2)} and the graph-theoretic RFD conditions for
ultragraphs reduce exactly to Bellier's graph conditions, we obtain the following consequence.

\begin{corollary}[The graph case]
\label{cor:graph-equivalence-algebraic-Cstar-RFD}
Let \(E\) be a graph and let \(K\) be a field.  Then the following are
equivalent:
\begin{enumerate}[(i)]
\item \(C^*(E)\) is RFD;
\item \(L_K(E)\) is RFD;
\item \(E\) satisfies Bellier's graph conditions.
\end{enumerate}
\end{corollary}

\section{RFD ultragraph algebras beyond graph algebras}\label{examplesec}

The purpose of this section is twofold.  First, we illustrate the
graph-theoretic RFD conditions by an example of an ultragraph which is genuinely
outside the graph setting.  Second, we show that this distinction is visible
not only at the level of the underlying combinatorial object, but also at the
level of the associated algebras.

Examples of ultragraph algebras which are not graph algebras are known in
both the \(C^*\)-algebraic and algebraic settings; see, for instance,
\cite{Tomforde2003Ultragraph} in the \(C^*\)-algebraic setting and
\cite{ImanfarPourabbasLarki2020UltragraphLPA,DuyenGoncalvesNam}
for ultragraph Leavitt path algebras.  The example below has the additional
feature that it belongs to the RFD class considered in this paper.  More
precisely, it satisfies Condition~\emph{(RFUM2)} and the graph-theoretic RFD
conditions, so both \(L_K(\mathcal G)\) and \(C^*(\mathcal G)\) are RFD.
Nevertheless, \(L_K(\mathcal G)\) is not isomorphic to the Leavitt path
algebra of any graph, and \(C^*(\mathcal G)\) is not isomorphic to the
\(C^*\)-algebra of any graph.

\begin{example} Let $G^0=\{v\}\cup\{w_n:n\in\mathbb N\} $ and let $ \mathcal G^1=\{e\}\cup\{e_n:n\in\mathbb N\}.$ Define $s(e)=v,$ $r(e)=A:=\{w_n:n\in\mathbb N\},$  and, for each \(n\in\mathbb N\), $ s(e_n)=w_n,$ $r(e_n)=\{w_n\}. $
\vspace{1.5cm}

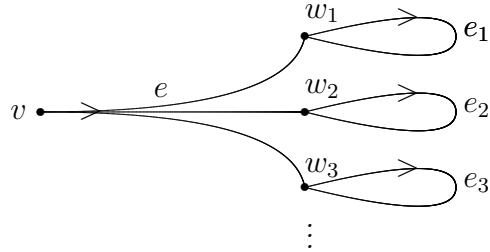
\centerline{
\setlength{\unitlength}{1cm}
\begin{picture}(5,0)
\put(-0.5,0){\circle*{0.1}}
\put(-0.9,-0.1){$v$}
\qbezier(-0.5,0)(2.8,0)(3,1)
\qbezier(-0.5,0)(2.8,0)(3,0)
\qbezier(-0.5,0)(2.8,0)(3,-1)
\put(3,1){\circle*{0.1}}
\put(3,0){\circle*{0.1}}
\put(3,-1){\circle*{0.1}}
\put(3,-1.8){\vdots}
\qbezier(3,1)(5,1.5)(5,1)
\qbezier(3,1)(5,0.5)(5,1)
\put(5.1,1){$e_1$}
\put(4.2,1.15){$>$}
\qbezier(3,0)(5,0.5)(5,0)
\qbezier(3,0)(5,-0.5)(5,0)
\put(5.1,0){$e_2$}
\put(4.2,0.15){$>$}
\qbezier(3,-1)(5,-0.5)(5,-1)
\qbezier(3,-1)(5,-1.5)(5,-1)
\put(5.1,-1){$e_3$}
\put(4.2,-0.85){$>$}
\put(3,1.25){$w_1$}
\put(3,0.25){$w_2$}
\put(3,-0.8){$w_3$}
\put(1,0.2){$e$}
\put(5.1,1){$e_1$}
\put(0,-0.1){$>$}
\end{picture}}
\vspace{2 cm}

It is straightforward to check that \(\mathcal G\) satisfies the
graph-theoretic RFD conditions.  Hence, by
Theorems~\ref{thm:algebraic-RFD-ultragraph} and
\ref{thm:Cstar-RFD-ultragraph}, both the Leavitt path algebra
\(L_K(\mathcal G)\) and the ultragraph \(C^*\)-algebra \(C^*(\mathcal G)\) are
RFD.
Moreover, \(r(e)\) is a minimal infinite emitter, while each \(r(e_n)\) is a
singleton regular vertex.  Thus \(\mathcal G\) also satisfies
Condition~\emph{(RFUM2)}.

\end{example}

Before proving that the algebras associated with the ultragraph above are not
graph algebras, we record two auxiliary facts about central idempotents and
central projections in graph algebras with finitely many vertices.

\begin{lemma}\label{lem:finite-central-idempotents-finite-vertex-graph-LPA}
Let \(E\) be a graph with \(E^0\) finite, and let \(K\) be a field.  Then
\(L_K(E)\) has only finitely many central idempotents.
\end{lemma}

\begin{proof}
Since \(E^0\) is finite, the unit space \(G_E^{(0)}\) of the graph groupoid is
compact: indeed,
\[
   G_E^{(0)}=\bigcup_{v\in E^0} Z(v),
\]
and each \(Z(v)\) is compact open.  The whole unit space is, of course,
invariant.

By \cite[Lemma~3.6]{ClarkMartinBarquero},
every compact open invariant subset \(U\subseteq G_E^{(0)}\) is a finite
disjoint union of the minimal compact open invariant subsets contained in
\(U\).  Applying this to \(U=G_E^{(0)}\), we obtain that there are only
finitely many minimal compact open invariant subsets of \(G_E^{(0)}\).

The center of \(L_K(E)\) is described in
\cite[Theorem~3.8]{ClarkMartinBarquero}
and in the structural decomposition following that theorem (\cite[Theorem~3.9]{ClarkMartinBarquero}).  More precisely,
the center is a finite direct sum of components of two types: copies of \(K\),
coming from minimal compact open invariant subsets with no cycle
contribution, and copies of \(K[x,x^{-1}]\), coming from no-exit cycle
components.

Let us indicate why only finitely many of these summands occur in the present
case.  The summands of the first type are indexed by a subset of the finite set
of minimal compact open invariant subsets of \(G_E^{(0)}\).  The summands of
the second type are indexed by no-exit cycles \(c\) for which the associated
set \(U_{c^0,\emptyset}\) is compact.  Since \(E^0\) is finite, there are only
finitely many such cycle classes: two no-exit cycles which meet at a vertex are
cyclic permutations of one another, and there are only finitely many vertices.

Thus \(Z(L_K(E))\) is a finite direct sum of copies of \(K\) and
\(K[x,x^{-1}]\).  Since both \(K\) and \(K[x,x^{-1}]\) have only the
idempotents \(0\) and \(1\), the center \(Z(L_K(E))\) has only finitely many
idempotents.  Equivalently, \(L_K(E)\) has only finitely many central
idempotents.
\end{proof}

\begin{remark}
    The above lemma may also be proved using the results in \cite{AlahmadiAlsulami}.

\end{remark}

\begin{lemma}\label{lem:finite-central-projections-finite-vertex-graph-Cstar}
Let \(E\) be a countable graph with \(E^0\) finite.  Then \(C^*(E)\) has
only finitely many central projections.
\end{lemma}

\begin{proof}
Since \(E^0\) is finite, \(C^*(E)\) is unital.   By
\cite[Proposition~3.4]{Eilers1999ConnectivityComponents}, central
projections of \(C^*(E)\) correspond to clopen subsets of
\(\operatorname{Prim}(C^*(E))\).  It is therefore enough to show that
\(\operatorname{Prim}(C^*(E))\) has only finitely many clopen subsets.

By \cite[Corollary~2.11]{HongSzymanski2004Primitive}, the primitive ideals
of \(C^*(E)\) are described by three types of parameters:
\[
   M_\gamma(E),\qquad BV(E),\qquad M_\tau(E)\times\mathbb T.
\]
Since \(E^0\) is finite, there are only finitely many subsets of \(E^0\).
Hence \(M_\gamma(E)\) and \(M_\tau(E)\) are finite.  Also
\(BV(E)\subseteq E^0\), so \(BV(E)\) is finite.

Thus \(\operatorname{Prim}(C^*(E))\) is covered by finitely many pieces:
the singleton pieces corresponding to \(M_\gamma(E)\) and \(BV(E)\), and,
for each \(N\in M_\tau(E)\), the circle fiber $
   X_N:=\{R_{N,t}:t\in\mathbb T\}.$
Although each \(X_N\) contains infinitely many primitive ideals, there are
only finitely many such fibers.

As observed immediately before \cite[Lemma~2.8]{HongSzymanski2004Primitive}, for each \(N\in M_\tau(E)\), the family
   $X_N:=\{R_{N,t}:t\in\mathbb T\}$
embeds topologically as a circle in \(\operatorname{Prim}(C^*(E))\).  Equivalently,
\(X_N\) is homeomorphic to \(\mathbb T\).  Hence each \(X_N\) is connected.

It follows that \(\operatorname{Prim}(C^*(E))\) is covered by finitely many
connected subspaces.  Each connected subspace is contained in a connected
component, so \(\operatorname{Prim}(C^*(E))\) has only finitely many
connected components.  Since every clopen subset of a topological space is a
union of connected components, \(\operatorname{Prim}(C^*(E))\) has only
finitely many clopen subsets.

Therefore \(C^*(E)\) has only finitely many central projections.
\end{proof}

\begin{remark}
For the topology on the primitive ideal space of a graph \(C^*\)-algebra, see
\cite{HongSzymanski2004Primitive}; see also
\cite{Gabe2013GraphT1Primitive} for a correction to
\cite[Theorem~3.4]{HongSzymanski2004Primitive}.
\end{remark}

With the two lemmas above in hand, we can distinguish the algebras associated with
\(\mathcal G\) from graph algebras by using the central idempotents and central
projections coming from the vertices \(w_n\).

\begin{proposition}
For the ultragraph \(\mathcal G\) in the preceding example, the Leavitt path
algebra \(L_K(\mathcal G)\) is not isomorphic to the Leavitt path algebra of
any countable graph.  Likewise, the ultragraph \(C^*\)-algebra
\(C^*(\mathcal G)\) is not isomorphic to the graph \(C^*\)-algebra of any
countable graph.
\end{proposition}

\begin{proof}
We use the same obstruction in both settings.  The algebras \(L_K(\mathcal G)\)
and \(C^*(\mathcal G)\) are unital, with unit \(p_v+p_A\).  For each
\(n\in\mathbb N\), set
\[
   q_n:=p_{w_n}+s_e p_{w_n}s_e^*.
\]
A direct check on the generators shows that each \(q_n\) is central.  In
\(L_K(\mathcal G)\), each \(q_n\) is a nonzero central idempotent, while in
\(C^*(\mathcal G)\), each \(q_n\) is a nonzero central projection.  Moreover,
the family \(\{q_n:n\in\mathbb N\}\) is pairwise orthogonal.  Hence
\(L_K(\mathcal G)\) contains infinitely many nonzero pairwise orthogonal
central idempotents, and \(C^*(\mathcal G)\) contains infinitely many nonzero
pairwise orthogonal central projections.

Suppose first that \(L_K(\mathcal G)\cong L_K(E)\) for some countable graph
\(E\).  Since \(L_K(\mathcal G)\) is unital, \(L_K(E)\) is unital, and hence
\(E^0\) is finite.  By
Lemma~\ref{lem:finite-central-idempotents-finite-vertex-graph-LPA},
\(L_K(E)\) has only finitely many central idempotents, and therefore cannot
contain infinitely many nonzero pairwise orthogonal central idempotents.  This
contradiction shows that \(L_K(\mathcal G)\) is not isomorphic to the Leavitt
path algebra of any countable graph.

Similarly, suppose that \(C^*(\mathcal G)\cong C^*(E)\) for some countable
graph \(E\).  Since \(C^*(\mathcal G)\) is unital, \(C^*(E)\) is unital, and
hence \(E^0\) is finite.  By
Lemma~\ref{lem:finite-central-projections-finite-vertex-graph-Cstar},
\(C^*(E)\) has only finitely many central projections, and therefore cannot
contain infinitely many nonzero pairwise orthogonal central projections.  This
contradiction shows that \(C^*(\mathcal G)\) is not isomorphic to the graph
\(C^*\)-algebra of any countable graph.
\end{proof}

\bibliographystyle{abbrv}
\bibliography{rfd_branching_refs_revised}

@incollection{Gabe2013GraphT1Primitive,
  author    = {Gabe, James},
  title     = {Graph {$C^*$}-algebras with a {$T_1$} primitive ideal space},
  booktitle = {Operator Algebra and Dynamics},
  series    = {Springer Proceedings in Mathematics \& Statistics},
  volume    = {58},
  pages     = {141--156},
  publisher = {Springer},
  address   = {Heidelberg},
  year      = {2013},
  doi       = {10.1007/978-3-642-39459-1_7}
}

@article{ClarkMartinBarquero,
  author  = {Clark, Lisa Orloff and Mart{\'i}n Barquero, Dolores and
             Mart{\'i}n Gonz{\'a}lez, C{\'a}ndido and
             Siles Molina, Mercedes},
  title   = {Using the {Steinberg} algebra model to determine the center of any
             {Leavitt} path algebra},
  journal = {Israel Journal of Mathematics},
  volume  = {230},
  number  = {1},
  pages   = {23--44},
  year    = {2019},
  doi     = {10.1007/s11856-018-1816-8}
}

@article{HongSzymanski2004Primitive,
  author  = {Hong, Jeong Hee and Szyma{\'n}ski, Wojciech},
  title   = {The primitive ideal space of the {$C^*$}-algebras of infinite graphs},
  journal = {Journal of the Mathematical Society of Japan},
  volume  = {56},
  number  = {1},
  pages   = {45--64},
  year    = {2004},
  doi     = {10.2969/jmsj/1191418695}
}

@article{Eilers1999ConnectivityComponents,
  author  = {Eilers, S{\o}ren},
  title   = {Connectivity and components for {$C^*$}-algebras},
  journal = {Mathematica Scandinavica},
  volume  = {84},
  number  = {1},
  pages   = {119--136},
  year    = {1999}
}

@misc{AlahmadiAlsulami,
  author        = {Alahmadi, Adel and Alsulami, Hamed},
  title         = {Centers of {Leavitt} path algebras and their completions},
  year          = {2015},
  note          = {arXiv:1507.07439 [math.RA]},
  eprint        = {1507.07439},
  archivePrefix = {arXiv},
  primaryClass  = {math.RA},
  doi           = {10.48550/arXiv.1507.07439}
}

@article{TascaGoncalves2022KMSUltragraphSinks,
  author  = {Tasca, Felipe Augusto and Gon{\c c}alves, Daniel},
  title   = {{KMS} states and continuous orbit equivalence for ultragraph shift spaces with sinks},
  journal = {Publicacions Matem{\`a}tiques},
  volume  = {66},
  number  = {2},
  pages   = {729--787},
  year    = {2022},
  doi     = {10.5565/PUBLMAT6622208},
  eprint  = {2003.05793},
  archivePrefix = {arXiv},
  primaryClass  = {math.OA}
}

@article{BratteliJorgensen1997,
  author  = {Bratteli, Ola and Jorgensen, Palle E. T.},
  title   = {Isometries, shifts, {Cuntz} algebras and multiresolution wavelet analysis of scale {$N$}},
  journal = {Integral Equations and Operator Theory},
  volume  = {28},
  number  = {4},
  pages   = {382--443},
  year    = {1997},
  doi     = {10.1007/BF01200379}
}

@article{Ozawa2004QWEP,
  author  = {Ozawa, Narutaka},
  title   = {About the {QWEP} conjecture},
  journal = {International Journal of Mathematics},
  volume  = {15},
  number  = {5},
  pages   = {501--530},
  year    = {2004},
  doi     = {10.1142/S0129167X04002417}
}

@article{Dadarlat2000NonnuclearAF,
  author  = {Dadarlat, Marius},
  title   = {Nonnuclear subalgebras of {AF} algebras},
  journal = {American Journal of Mathematics},
  volume  = {122},
  number  = {3},
  pages   = {581--597},
  year    = {2000},
  doi     = {10.1353/ajm.2000.0018}
}

@article{GoncalvesLiRoyer2018HigherRankGraph,
  author  = {Gon{\c c}alves, Daniel and Li, Hui and Royer, Danilo},
  title   = {Branching systems for higher-rank graph {$C^*$}-algebras},
  journal = {Glasgow Mathematical Journal},
  volume  = {60},
  number  = {3},
  pages   = {731--751},
  year    = {2018},
  doi     = {10.1017/S0017089517000072}
}

@misc{BellierShulman2026,
  author        = {Bellier, Guillaume and Shulman, Tatiana},
  title         = {Decomposition theorems for unital graph {$C^*$}-algebras},
  year          = {2026},
  note          = {arXiv:2505.12769 [math.OA]},
  eprint        = {2505.12769},
  archivePrefix = {arXiv},
  primaryClass  = {math.OA},
  doi           = {10.48550/arXiv.2505.12769}
}

@misc{Bellier2026,
  author        = {Bellier, Guillaume},
  title         = {The {RFD} property for graph {$C^*$}-algebras},
  year          = {2026},
  note          = {arXiv:2604.06993 [math.OA]},
  eprint        = {2604.06993},
  archivePrefix = {arXiv},
  primaryClass  = {math.OA},
  doi           = {10.48550/arXiv.2604.06993}
}

@article{Archbold1995,
  author  = {Archbold, Robert J.},
  title   = {On residually finite-dimensional {$C^*$}-algebras},
  journal = {Proceedings of the American Mathematical Society},
  volume  = {123},
  number  = {9},
  pages   = {2935--2937},
  year    = {1995},
  doi     = {10.1090/S0002-9939-1995-1257119-5}
}

@article{ExelLoring1992,
  author  = {Exel, Ruy and Loring, Terry A.},
  title   = {Finite-dimensional representations of free product {$C^*$}-algebras},
  journal = {International Journal of Mathematics},
  volume  = {3},
  number  = {4},
  pages   = {469--476},
  year    = {1992},
  doi     = {10.1142/S0129167X92000217}
}

@article{Farkas2005,
  author  = {Farkas, Daniel R.},
  title   = {Generalized reductive algebras and a quantum example},
  journal = {Pacific Journal of Mathematics},
  volume  = {221},
  number  = {1},
  pages   = {35--57},
  year    = {2005},
  doi     = {10.2140/pjm.2005.221.35}
}

@article{GoncalvesLiRoyer2016,
  author  = {Gon{\c c}alves, Daniel and Li, Hui and Royer, Danilo},
  title   = {Faithful representations of graph algebras via branching systems},
  journal = {Canadian Mathematical Bulletin},
  volume  = {59},
  number  = {1},
  pages   = {95--103},
  year    = {2016},
  doi     = {10.4153/CMB-2015-023-6}
}

@article{Reyes2024FiniteDual,
  author  = {Reyes, Manuel L.},
  title   = {The finite dual coalgebra as a quantization of the maximal spectrum},
  journal = {Journal of Algebra},
  volume  = {644},
  pages   = {287--328},
  year    = {2024},
  doi     = {10.1016/j.jalgebra.2023.12.035}
}

@article{Rowen2019HopfianBassian,
  author        = {Rowen, Louis H.},
  title         = {Hopfian and {Bassian} algebras},
  journal       = {Communications in Algebra},
  volume        = {47},
  number        = {9},
  pages         = {3757--3766},
  year          = {2019},
  doi           = {10.1080/00927872.2019.1576182},
  eprint        = {1711.06483},
  archivePrefix = {arXiv},
  primaryClass  = {math.RA}
}

@article{LarsenShalev2022RFD,
  author  = {Larsen, Michael and Shalev, Aner},
  title   = {Residually finite dimensional algebras and polynomial almost identities},
  journal = {Journal of Algebra and Its Applications},
  volume  = {21},
  number  = {2},
  pages   = {2250038},
  year    = {2022},
  doi     = {10.1142/S0219498822500384}
}

@article{ShulmanSkalski,
  author  = {Shulman, Tatiana and Skalski, Adam},
  title   = {{RFD} property for groupoid {$C^*$}-algebras of amenable groupoids and for crossed products by amenable actions},
  journal = {Journal of Functional Analysis},
  volume  = {291},
  number  = {3},
  pages   = {111510},
  year    = {2026},
  doi     = {10.1016/j.jfa.2026.111510},
  eprint  = {2305.12122},
  archivePrefix = {arXiv},
  primaryClass = {math.OA}
}

@article{GoncalvesRoyer2019InfiniteAlphabetEdgeShift,
  author  = {Gon{\c c}alves, Daniel and Royer, Danilo},
  title   = {Infinite alphabet edge shift spaces via ultragraphs and their {$C^*$}-algebras},
  journal = {International Mathematics Research Notices},
  volume  = {2019},
  number  = {7},
  pages   = {2177--2203},
  year    = {2019},
  doi     = {10.1093/imrn/rnx187},
  eprint  = {1703.05069},
  archivePrefix = {arXiv},
  primaryClass  = {math.OA}
}

@article{EidtGoncalvesRoyerRelativeUltragraph,
  author        = {Eidt, Ben-Hur and Gon{\c c}alves, Daniel and Royer, Danilo},
  title         = {Relative ultragraph algebras and infinite interval maps},
  journal       = {Indiana University Mathematics Journal},
  note          = {To appear. arXiv:2508.19835 [math.OA]},
  year    = {2026},
  eprint        = {2508.19835},
  archivePrefix = {arXiv},
  primaryClass  = {math.OA},
  doi           = {10.48550/arXiv.2508.19835}
}

@article{UltragraphGroupoids,
  author  = {de Castro, Gilles G. and Gon{\c c}alves, Daniel and van Wyk, Daniel W.},
  title   = {Ultragraph algebras via labelled graph groupoids, with applications to generalized uniqueness theorems},
  journal = {Journal of Algebra},
  volume  = {579},
  pages   = {456--495},
  year    = {2021},
  doi     = {10.1016/j.jalgebra.2021.04.002}
}

@article{ImanfarPourabbasLarki2020UltragraphLPA,
  author  = {Imanfar, M. and Pourabbas, A. and Larki, H.},
  title   = {The {Leavitt} path algebras of ultragraphs},
  journal = {Kyungpook Mathematical Journal},
  volume  = {60},
  number  = {1},
  pages   = {21--43},
  year    = {2020}
}

@article{GoncalvesRoyer2011,
  author  = {Gon{\c c}alves, Daniel and Royer, Danilo},
  title   = {On the representations of {Leavitt} path algebras},
  journal = {Journal of Algebra},
  volume  = {333},
  number  = {1},
  pages   = {258--272},
  year    = {2011},
  doi     = {10.1016/j.jalgebra.2011.02.009}
}

@article{GoncalvesRoyer2012,
  author  = {Gon{\c c}alves, Daniel and Royer, Danilo},
  title   = {Graph {$C^*$}-algebras, branching systems and the {Perron--Frobenius} operator},
  journal = {Journal of Mathematical Analysis and Applications},
  volume  = {391},
  number  = {2},
  pages   = {457--465},
  year    = {2012},
  doi     = {10.1016/j.jmaa.2012.02.048}
}

@article{KraftSmallWallach2001,
  author  = {Kraft, Hanspeter and Small, Lance W. and Wallach, Nolan R.},
  title   = {Properties and examples of {FCR}-algebras},
  journal = {Manuscripta Mathematica},
  volume  = {104},
  pages   = {443--450},
  year    = {2001},
  doi     = {10.1007/s002290170018}
}

@article{Szymanski2002,
  author  = {Szyma\'{n}ski, Wojciech},
  title   = {General {Cuntz--Krieger} uniqueness theorem},
  journal = {International Journal of Mathematics},
  volume  = {13},
  number  = {5},
  pages   = {549--555},
  year    = {2002},
  doi     = {10.1142/S0129167X0200137X}
}

@article{Tomforde2003Ultragraph,
  author  = {Tomforde, Mark},
  title   = {A unified approach to {Exel--Laca} algebras and {$C^*$}-algebras associated to graphs},
  journal = {Journal of Operator Theory},
  volume  = {50},
  number  = {2},
  pages   = {345--368},
  year    = {2003}
}

@article{DuyenGoncalvesNam,
  author  = {Duyen, T. T. H. and Gon{\c c}alves, Daniel and Nam, T. G.},
  title   = {On the ideals of ultragraph {Leavitt} path algebras},
  journal = {Algebras and Representation Theory},
  volume  = {27},
  number  = {1},
  pages   = {77--113},
  year    = {2024},
  doi     = {10.1007/s10468-023-10206-0}
}

@article{Tomforde2003SimplicityUltragraph,
  author  = {Tomforde, Mark},
  title   = {Simplicity of ultragraph algebras},
  journal = {Indiana University Mathematics Journal},
  volume  = {52},
  number  = {4},
  pages   = {901--925},
  year    = {2003}
}

@article{GoncalvesLiRoyer2016Ultragraph,
  author  = {Gon{\c c}alves, Daniel and Li, Hui and Royer, Danilo},
  title   = {Branching systems and general {Cuntz--Krieger} uniqueness theorem for ultragraph {$C^*$}-algebras},
  journal = {International Journal of Mathematics},
  volume  = {27},
  number  = {10},
  pages   = {1650083},
  year    = {2016},
  doi     = {10.1142/S0129167X1650083X}
}

@article{GoncalvesRoyer2021UltragraphReduction,
  author  = {Gon{\c c}alves, Daniel and Royer, Danilo},
  title   = {Representations and the reduction theorem for ultragraph {Leavitt} path algebras},
  journal = {Journal of Algebraic Combinatorics},
  volume  = {53},
  pages   = {505--526},
  year    = {2021},
  doi     = {10.1007/s10801-020-01004-8}
}

\end{document}